\documentclass[a4paper]{article}
\usepackage[dvipsnames]{xcolor}
\usepackage[inline]{enumitem}
\usepackage{mathtools}

\usepackage[utf8]{inputenc}
\usepackage[T1]{fontenc}
\usepackage{amsmath,amsthm,amssymb,amsfonts,subcaption}
\usepackage{xspace}
\usepackage{hyperref}

\usepackage{verbatim}
\usepackage{tikz}
\usepackage{setspace}
\usepackage{cite}
\usetikzlibrary{snakes}
\usepackage[procnumbered,linesnumbered,ruled,vlined]{algorithm2e}
\usetikzlibrary{decorations.pathmorphing}

\hypersetup{
	colorlinks=true,breaklinks,
	linktoc=section,
	linkcolor=Maroon,
	linkbordercolor=white,
	citecolor=dartmouthgreen,
	urlcolor=dartmouthgreen,
	pdfborder = {0 0 1}
}

\usetikzlibrary{positioning, fadings, backgrounds}

\usepackage[textsize=footnotesize,color=green!40]{todonotes}
\usepackage[hmargin=2.5cm,vmargin=3cm]{geometry}

\usepackage[nameinlink]{cleveref}
\Crefname{algocf}{algorithm}{Procedure}

\usetikzlibrary{decorations.pathmorphing}
\tikzset{snake it/.style={decorate, decoration=snake}}

\newenvironment{subproof}[1][\proofname]{%
	\begin{proof}[#1]%
	}{%
	\end{proof}%
}

\definecolor{dartmouthgreen}{rgb}{0.05, 0.5, 0.06}

\newtheorem{theorem}{Theorem}

\newtheorem{proposition}{Proposition}
\newtheorem{lemma}[theorem]{Lemma}
\newtheorem{observation}{Observation}
\newtheorem{corollary}{Corollary}

\newtheorem{definition}{Definition}

\newtheorem{conjecture}[theorem]{Conjecture}
\newtheorem{claim}{Claim}[theorem]

\newcommand{\revise}[1]{{\color{black} #1}}
\title{ $K_{2,3}$-induced minor-free graphs admit quasi-isometry with additive distortion to graphs of tree-width at most two}
\date{}	

\author{Dibyayan Chakraborty\footnote{School of Computer Science, University of Leeds, United Kingdom. \texttt{Email: D.chakraborty@leeds.ac.uk}}}

\newcommand{\Ball}[2]{B_{#1}\left({#2}\right)}
\newcommand{\Sphere}[2]{N_{#1}(#2)}
\newcommand{\clusterGraphv}[2]{ \Gamma \left({#1},{#2}\right) }

\newcommand{\distance}[3]{d_{#1}\left(#2,#3\right) }

\newcommand{\ParentSet}[2]{ P\left(#1, #2\right)}

\newcommand{\InducedSub}[2]{#1\left[#2\right]}

\newcommand{\dist}[2]{d\left(#1,#2\right)}

\newcommand{\dmod}[1]{\left|#1\right|}

\newcommand{\sipco}[1]{sipco\left(#1\right)}

\newcommand{\ipco}[1]{ipco\left(#1\right)}

\begin{document}
	
	\maketitle
	
	\begin{abstract}
		A graph $H$ is an \emph{induced minor} of a graph $G$ if $H$ can be obtained from $G$ by a sequence of edge contractions and vertex deletions. Otherwise, $G$ is \emph{$H$-induced minor-free}. In this paper, we provide a different proof of the fact that $K_{2,3}$-induced minor-free graphs admit a quasi-isometry with additive distortion to graphs with tree-width at most two. Our proof yields a $O(nm)$-time algorithm which takes as input a $K_{2,3}$-induced minor-free graph with $n$ vertices and $m$ edges, and outputs a tree-width two graph $H$ with the desired additive distortion. For \emph{universally signable} graphs, a subclass of $K_{2,3}$-induced minor-free graphs, the time complexity of our algorithm is linear. As a consequence, we  obtain a truly sub-quadratic time additive constant factor approximation algorithm to compute the \emph{diameter} of a universally signable graph. In contrast, assuming the \emph{Strong Exponential Time Hypothesis} (\textsc{SETH}), the diameter of split graphs (a very restricted class of universally signable graphs), cannot be computed in truly sub-quadratic time [Borassi et al. (ENTCS, 2016)].  
	\end{abstract}

	\section{Introduction}
	
	
	When is a graph \emph{quasi-isometric} (\Cref{def:quasi}) to a much simpler \emph{host} graph? This question is central to \emph{coarse graph theory}~\cite{georgakopoulos2023graph}. When the host graphs are restricted to the class of graphs with bounded \emph{tree-width}, remarkably, the answer is provided by Nguyen, Scott \& Seymour~\cite{nguyen2025coarse} and independently by Hickingbotham~\cite{hickingbotham2025graphs}. Specifically, they proved the following:
	
	\begin{theorem}[\cite{nguyen2025coarse,hickingbotham2025graphs}]
		A graph is quasi-isometric to a graph with bounded tree-width if and only if it has a tree-decomposition where each bag consists of a bounded number of balls of bounded diameter.
	\end{theorem}
	
	\sloppy When quasi-isometric maps with only \emph{additive distortions} (\Cref{def:add}) are allowed, the situation is much less understood. Brandst{\"a}dt et al.~\cite{brandstadt1999distance} proved that \emph{chordal} graphs (i.e., graphs without any induced cycle of length greater than three) admit a quasi-isometry to trees with additive distortion at most $2$. This result has been later generalized to \emph{$k$-chordal} graphs~\cite{chepoi2000note}, and more generally to graphs of bounded \emph{treelength}~\cite{dourisboure2007tree}.  Berger \& Seymour~\cite{berger2024bounded} characterized graphs admitting quasi-isometry with additive distortion to trees. See~\cite{dragan2025graph} for a comprehensive overview of graphs admitting a quasi-isometry with additive distortion to trees.
	
	\emph{When does a graph class admit quasi-isometry with additive distortion to graphs with tree-width at most two?} The above question also relates to the following conjecture of Nguyen, Scott \& Seymour~\cite{nguyen2025coarse}.
	
	\begin{conjecture}[\cite{nguyen2025coarse}]\label{conj}
		There is a constant $k$, such that if a graph $G$ admits a quasi-isometry to a graph of tree-width at most two, then $G$ admits a quasi-isometry with additive distortion to a graph with tree-width at most $k$. 
	\end{conjecture}

	In this paper, we consider the class of \emph{$K_{2,3}$-induced minor-free} graphs. A graph $H$ is an \emph{induced minor} of a graph $G$ if $H$ can be obtained from $G$ by a sequence of edge contractions and vertex deletions. Otherwise, $G$ is \emph{$H$-induced minor-free}. The class of $K_{2,3}$-induced minor-free graphs have been studied extensively in structural graph theory~\cite{DBLP:conf/iwoca/DallardDHMPT24,dallard2024treewidth}. $K_{2,3}$-induced minor-free graphs generalize many popular graph classes. Examples include outerplanar graphs, circular-arc graphs~\cite{francis2014blocking}, chordal graphs, universally signable graph~\cite{conforti1997universally} etc. Coarse generalizations of $K_{2,3}$-induced minor-free graphs have been studied~\cite{chepoi2012constant}. A recent result of Albrechtsen et al.~\cite{albrechtsen2024characterisation} implies that $K_{2,3}$-induced minor-free graphs are quasi-isometric to graphs with tree-width at most two. Combining results of Fujiwara \& Papasoglu~\cite{fujiwara2023coarse} and Nguyen, Seymour \& Scott~\cite{nguyen2025coarse} prove that \Cref{conj} holds for $K_{2,3}$-induced minor-free graphs\footnote{We thank Robert Hickingbotham for informing us of this connection.}.	In this paper, we provide a different constructive proof for the above. Specifically, we prove the following theorem.

	\begin{theorem}\label{main:th:K23}
		There is a constant $c$, and an $O(nm)$-time algorithm that takes as input a $K_{2,3}$-induced minor-free graph $G$ with $n$ vertices and $m$ edges, and outputs a graph $H$ with $V(G)=V(H)$, tree-width at most two, and $\dmod{\distance{G}{u}{v} - \distance{H}{u}{v}}\leq c$ for any $u,v\in V(G)$. 
	\end{theorem}

	In the above theorem, $\distance{G}{.}{.}$ and  $\distance{H}{.}{.}$ denote the standard graph metric in $G$ and $H$, respectively. 
	\revise{The class of universally-signable graphs~\cite{conforti1997universally} is an important subclass of $K_{2,3}$-induced minor-free graphs. A graph $G$ is \emph{universally-signable} if for all choices of vector $\gamma$, there exists a subset $F$ of the edge set of $G$ such that $|F \cap H| \equiv \gamma_H (\mod 2)$, for all \emph{holes} $H$ of $G$. A \emph{hole} of $G$ is an induced subgraph of $G$ isomorphic to $C_k$ with $k\geq 4$. A characterization of universally-signable graphs via {forbidden induced subgraphs} states that a graph is universally signable if and only if it has none of the four \emph{Truemper configurations} as induced subgraphs. Conforti et al.~\cite{conforti1997universally} proved that a graph is universally-signable if and only if it is a clique or a \emph{hole} (i.e. an induced cycle of length more than three) or it has a \emph{clique cutset}. This result generalises a famous theorem of Dirac~\cite{dirac1961rigid} (on \emph{chordal graphs}) and is one of the earliest instance of \emph{decomposition theorems} for hereditary graph classes, a technique which were instrumental in the resolution of the Strong Perfect Graph Conjecture~\cite{chudnovsky2006strong}.}
	
	When the input graphs are restricted to universally signable graphs, a subclass of $K_{2,3}$-induced minor-free graph, our algorithm runs in linear time, which  has the following consequence. 
	The \emph{diameter} of a graph is the length of a longest isometric path (i.e. a path of minimum length between the end-vertices) between two of its vertices. Given an undirected graph $G$, \textsc{Diameter} is the algorithmic task of computing the diameter of $G$. \textsc{Diameter} is a fundamental problem	with countless applications in computer science. The fastest combinatorial algorithm for \textsc{Diameter} has a running time of $O(nm)$ in $m$-edge, $n$-vertex graphs. Surprisingly, Roditty and Vassilevska Williams~\cite{roditty2013fast} proved that a truly sub-quadratic $\left(\frac{3}{2}-\epsilon \right)$-approximation algorithm for \textsc{Diameter} would falsify the \emph{Strong Exponential Time Hypothesis}~\cite{impagliazzo2001problems}. This motivated researchers to study the \emph{approximability} of \textsc{Diameter} on special graph classes. However,  Borassi, Crescenzi \& Habib~\cite{borassi2016into} proved that assuming SETH, diameter  cannot be computed optimally in truly sub-quadratic time even on split graphs, i.e. graphs whose vertex set can be partitioned into a clique and an independent set. 
	Split graphs are universally signable graphs. Hence assuming SETH, the diameter of universally signable graphs cannot be computed optimally in truly sub-quadratic time. The diameter of graphs with tree-width at most two can be computed in truly sub-quadratic time~\cite{cabello2018subquadratic}. Hence, \Cref{main:th:K23} imply the following corollary.
	
	\begin{corollary}\label{cor:univ}
		\textsc{Diameter} admits a truly sub-quadratic additive constant factor approximation algorithm on universally signable graphs.
	\end{corollary}
	
	\revise{Note that the above result cannot be obtained using existing results from the literature. For example, the additive approximation algorithm for \textsc{Diameter} on \emph{hyperbolic} graphs~\cite{chepoi2008diameters} cannot be applied to universally-signable graphs, as they may arbitrarily long isometric cycles (and thus not hyperbolic).}
	Below, we provide an informal description of the tools used in the proofs of \Cref{main:th:K23}.  Then,	we provide an overview of the proofs.
	
	\medskip \noindent \textbf{Important tools:} The proof of \Cref{main:th:K23} uses \emph{layering partition}~\cite{brandstadt1999distance,chepoi2000note} and the notion of \emph{strong isometric path complexity}~\cite{c2023isometric} of graphs. 
	Layering partitions have been used to solve numerous problems in metric graph
	theory~\cite{brandstadt1999distance,chepoi2000note,chepoi2012constant,dragan2025graph}. See \Cref{sec:layer} for definitions. Informally, in 
	a layering partition, the vertex set of a graph is partitioned into \emph{clusters} w.r.t a fixed \emph{root} vertex. Two vertices are in the same cluster if they are at a distance $k$ from the root for some $k\geq 0$ and are end-vertices of a path that is disjoint from the ball of radius $k-1$ centered at 
	the root vertex. It was shown in~\cite{brandstadt1999distance,chepoi2000note} that the adjacency relation between the clusters is captured by a tree, known as the \emph{layering tree}. Moreover, for a fixed root vertex, the corresponding layering tree can be constructed in linear time~\cite{chepoi2000note}. 
	
	The \emph{strong isometric path complexity} is a recently introduced graph invariant that captures how arbitrary isometric paths of a 
	graph can be viewed as a union of few ``rooted" isometric paths (i.e. {isometric} paths with a common end-vertex). See \Cref{def:strong}. Informally, if the strong isometric path complexity of a graph $G$ is at most $k$, then for any isometric path $P$ of $G$ and any vertex of $v\in V(G)$, the vertices of $P$ can be covered with $k$-many $v$-rooted isometric paths of $G$. 	Originally, strong isometric path complexity and its weaker variant isometric path complexity were introduced to propose \emph{approximation algorithms} for the problem \textsc{Isometric Path Cover}, where the objective is
	to find a minimum-cardinality set of isometric paths so that each vertex of the graph belongs to at least one of the solution paths. \textsc{Isometric Path 
		Cover} has been studied recently in the algorithmic graph theory community~\cite{c22isom,c2023isometric,MFCS2024}. It was shown that 
	\textsc{Isometric Path Cover} admit a polynomial-time $c$-approximation algorithm on graphs with strong isometric path complexity at most 
	$c$~\cite{c2023isometric} (while being NP-hard even on chordal graphs whose strong isometric path complexity is at most 
	four~\cite{c22isom,c2023isometric}). From  the structural perspective,   many important graph classes such as {\emph{hyperbolic} graphs~\cite{gromov1987}, 
		\emph{outerstring} graphs~\cite{rok2019outerstring, bose2022computing}, and \emph{(theta,prism,pyramid)}-free graphs~\cite{abrishami2022graphs,chudnovsky2024tree} have bounded strong isometric path
		complexity~\cite{c2023isometric}. It was shown in~\cite{c2023isometric, chakraborty2025strong} that $K_{2,3}$-induced minor-free graphs have bounded strong isometric path complexity. We shall use this fact in the proof of \Cref{main:th:K23}. On the other hand, graphs with tree-width at most two have unbounded strong isometric path complexity.

		\medskip \noindent \textbf{Overview of the proofs of \Cref{main:th:K23}:}
		Given a $K_{2,3}$-induced minor free graph $G$, we choose a vertex $r\in V(G)$ arbitrarily as the root and construct the layering tree w.r.t $r$, denoted by $\clusterGraphv{r}{G}$. Let $R=\{r\}$ be the \emph{root cluster}. When $\clusterGraphv{r}{G}$ is rooted at $R$, there is a natural \emph{parent-child} relationship among the clusters. For a cluster $S$, the \emph{parent set} of $S$, denoted as $\ParentSet{r}{S}$, is the set of vertices that lie in the parent of $S$ \revise{in $\clusterGraphv{r}{G}$} and are adjacent to the vertices in $S$. We observe that given a cluster $S$ whose parent is not $R$, $\ParentSet{r}{S}$  is a \emph{minimal cutset} of the original input graph. The size of the largest stable set in the graph induced by a minimal cutset of $K_{2,3}$-induced minor-free graph is at most two~\cite{dallard2024treewidth}. Hence, for any cluster $S$, the size of the largest stable set in the graph induced by $\ParentSet{r}{S}$ in $G$ is at most two. 
		
		We begin our construction by initializing a graph $H=(V(G),\emptyset)$. We traverse $\clusterGraphv{r}{G}$ (rooted at $R$) bottom-up and process the clusters in $G$ accordingly. For each cluster $S$, depending on the structure of the graph induced by $\ParentSet{r}{S}$, we introduce edges in $H$ between the vertices of $S$ and at most two special vertices of $\ParentSet{r}{S}$. See \Cref{algo:embed} for details of the construction of $H$. Note that the final resulting graph $H$ may not be a subgraph of $G$. 
		
		To prove the upper bound on the distortion, we show that $H$ is connected, and for any edge $uv\in E(G)$, $\distance{H}{u}{v}\leq 16$. Next, we show that for any $(u,v)$-isometric path, which is a subpath of some $r$-rooted isometric path in $G$, we have $\distance{H}{u}{v}\leq \distance{G}{u}{v}+20$. Now consider any $(x,y)$-isometric path $P$ in $G$. Since the strong isometric path complexity of $G$ is bounded, the edges of $P$ can be partitioned into a bounded number of edges and a bounded number of isometric paths that are subpaths of $r$-rooted isometric paths in $G$. Now, using our earlier observations, we show that there exists a constant $c$ (independent of $G$), such that  $\distance{H}{x}{y}\leq \distance{G}{x}{y} + c$. Using a similar technique, we show that for any $u,v\in V(G)$, $\distance{G}{u}{v}\leq \distance{H}{u}{v}+c$.  
		
		To prove the upper bound on the tree-width of $H$, we show that  $H$ is triangle-free and does not contain \emph{wheel} or \emph{theta} as induced subgraphs. See \Cref{sec:k23} for definitions. The absence of these structures in $H$ implies that $H$ is a universally signable graph, and has tree-width at most two. This concludes the proof of \Cref{main:th:K23}.

		\medskip \noindent \textbf{Organisation:} In \Cref{sec:prelim}, we introduce some notations, recall some existing results, and prove some preliminary observations. In \Cref{sec:proof}, we prove \Cref{main:th:K23} and conclude in \Cref{sec:conclude}.

		\section{Preliminaries and some observations}\label{sec:prelim}
		
		
		All graphs in this paper are finite, undirected, unweighted, and simple. Let $G$ be a graph. For a set $S\subseteq V(G)$, $\InducedSub{G}{S}$ denotes the subgraph of $G$ induced by $S$, and $G-S =\InducedSub{G}{V(G)\setminus S}$. When $S=\{w\}$, we write $G-\{w\}$ as $G-w$. For two graphs $G,H$, we write $H\subseteq G$, if $H$ is a subgraph of $G$. A \emph{stable set} of a graph is a set of pairwise non-adjacent vertices. 
		
		A \emph{cutset} of a graph $G$ is a set of vertices of $G$ whose removal disconnects $G$. For two vertices $u,v\in V(G)$, a \emph{$(u,v)$-cutset} of a graph $G$ is a set of vertices of $G$ whose removal disconnects $u,v$ in $G$. A cutset $S$ of a graph $G$ is \emph{minimal} if no proper subset of $S$ is a cutset of $G$. Similarly, for two vertices $u,v\in V(G)$, a $(u,v)$-cutset $S$ of a graph $G$ is \emph{minimal} if no proper subset of $S$ is a $(u,v)$-cutset of $G$.

		We state the definition of \emph{tree-decompositions} for completeness.	A \emph{tree-decomposition} of a graph $G$ is a rooted tree $T$ 
		where each node $v$ is associated to a subset $X_v$ of $V(G)$ called \emph{bag}, such that the set of nodes of $T$ containing a given vertex of $G$ forms a nonempty connected subtree of $T$, and any two adjacent vertices of $G$ appear in a common node of $T$.
		The \emph{width} of $T$ is the maximum cardinality of a bag minus $1$. The \emph{tree-width} of $G$ is the 	minimum integer $k$ such that $G$ has a tree-decomposition of width $k$.

		\subsection{$K_{2,3}$-induced minor-free graphs}	\label{sec:k23}

		We shall use the following result of Dallard, Milani{\v{c}}, and {\v{S}}torgel~\cite{dallard2024treewidth}.

		\begin{proposition}[\cite{dallard2024treewidth}]\label{prp:ind2}
			Let $G$ be a $K_{2,3}$-induced minor-free graph. For any two vertices $u,v\in V(G)$, the size of the largest stable set of the graph induced by any minimal $(u,v)$-cutset of $G$ is at most two. 
		\end{proposition}
		
		$K_{2,3}$-induced minor-free graphs can be also characterized by forbidding the  \emph{Truemper} configurations and its variants. A \emph{hole} is a cycle of length at least four. For an integer $t\geq 1$, a \emph{prism} is a graph made of three vertex-disjoint induced paths $P_1 \coloneq a_1\ldots b_1$, $P_2 \coloneq a_2\ldots b_2$, $P_3 \coloneq a_3\ldots b_3$ of lengths \revise{at least $1$}, such that $a_1 a_2 a_3$ and $b_1 b_2b_3$ are triangles and no edges exist between the paths except those of the two triangles. A
		prism is \emph{long} if at least one of its three paths has a length of at least 2. A \emph{pyramid} is a graph made of three induced paths $P_1 \coloneq a\ldots b_1$, $P_2 \coloneq a\ldots b_2$, $P_3 \coloneq a\ldots b_3$, such that $(i)$ {they are pairwise} vertex-disjoint except at $a$, $(ii)$ two of which are of lengths at least two, $(iii)$ $b_1 b_2 b_3$ is a triangle, and $(iv)$ no edges exist between the paths except those of the triangle and the three edges incident to $a$. A \emph{theta} is a graph made of three {internally} vertex-disjoint induced paths $P_1 \coloneq a\ldots b$, $P_2 \coloneq a\ldots b$, $P_3\coloneq a\ldots b$ of lengths at least two, and such that no edges exist between the paths except the three edges incident to $a$ and the three edges incident to $b$. A \emph{wheel} $(H, x)$ consists of a hole $H$, called the \emph{rim}, and a vertex $x$, called the \emph{center}, that has at least three neighbors on $H$. A \emph{sector} of a wheel $(H, x)$ is a path $P$ that is contained in $H$, whose end-vertices are neighbors of $x$ and whose internal vertices are not. Every wheel has at least three sectors. A wheel is \emph{broken} if
		at least two of its sectors have a length of at least 2. Dallard et al.~\cite{DBLP:conf/iwoca/DallardDHMPT24} proved the following theorem. 
		
		\begin{proposition}[\cite{DBLP:conf/iwoca/DallardDHMPT24}]\label{prp:induced-truem}
			A graph is $K_{2,3}$-induced minor-free if and only if it does not contain a long prism, pyramid, theta, or a broken wheel.
		\end{proposition}

		
		\emph{Universally signable} graphs form a subclass of $K_{2,3}$-induced minor-free graphs. See~\cite{conforti1997universally} for the original definition of universally signable   graphs. A graph is universally signable if and only if it does not contain any theta, pyramid, prism, or wheel~\cite{conforti1997universally}. Important graph classes like chordal graphs, outerplanar graphs are subclasses of universally signable graphs. Any minimal cutset in a universally signable graph G is either a clique or a stable set of size two~\cite{conforti1997universally}.  The next proposition follows from the proof of the above result. 
		
		\begin{proposition}[\cite{conforti1997universally}]\label{prp:decom}
			Let $C$ be a \revise{minimal} cutset of a universally signable graph $G$ and \revise{$G_1,G_2$ be two connected components of $G-C$}. For each $i\in \{1,2\}$ and each $v\in C$, there is a vertex $u_i\in V(G_i)$ with $vu_i\in E(G)$. Then  $\InducedSub{G}{C}$ is  either a clique or a stable set of size two. Moreover, if $\InducedSub{G}{C}$ is a stable set of size two, then $V(G)=V(G_1)\cup V(G_2)\cup C$. 
		\end{proposition}
		
		We shall also use the following lemma.
		
%
		
		\begin{lemma}\label{prp:tree-width}
			The tree-width of a triangle-free universally signable graph is at most two.
		\end{lemma}
		\begin{proof}
			\revise{We shall use the result of Conforti et al.~\cite{conforti1997universally} which states that a graph is universally-signable if and only if it is a clique or a \emph{hole} (i.e. an induced cycle of length more than three) or it has a \emph{clique cutset} (i.e. a clique whose removal disconnects the graph). We shall also use the notion of \emph{clique cutset decomposition tree} of a graph. When a graph $G$ has a clique cutset $K$, $V(G)$ can be partitioned into three sets $A,K,B$ such that $A\neq \emptyset, B\neq \emptyset$ and there is no edge between a vertex in $A$ and a vertex in $B$. 
			Such a triple $(A,K,B)$ is called a \emph{split} of clique cutset $K$. 
			The \emph{blocks of decomposition} with respect to a a split $(A,K,B)$ are the graphs $G_A=G[A\cup K]$ and $G_B=G[B\cup K]$. 
			A \emph{clique cutset decomposition tree} for a graph $G$ is a rooted tree $T$ defined
			as follows.
			\begin{itemize}
				\item The root of $T$ is $G$.
				\item Every non-leaf node of $T$ is a graph $G'$ that contains a clique cutset
				$K$ with split $(A, K, B)$ and the children of $G'$
				in $T$ are the blocks of
				decomposition $G'_A$ and $G'_B$  of $G'$ with respect to $(A, K, B)$.
				\item Every leaf of $T$ is a graph with no clique cutset.
			\end{itemize}
			If a triangle-free universally signable graph has no clique cutset, i.e. the corresponding clique cutset decomposition tree has only one node, then clearly the statement of the lemma is true. For an integer $t\geq 1$, let the statement is true for all triangle-free graphs whose clique cutset decomposition tree has at most $n$ nodes. Let $G$ be a triangle-free universally-signable graph with a clique cutset decomposition tree $T$ containing $n+1$ nodes. Let $G_A$ and $G_B$ are the two children of the root node of $T$ and $(A,K,B)$ be the corresponding split. Observe that, $G_A$ and $G_B$ both have clique cutset decomposition trees with at most $n$ nodes. Hence, $G_A$ and $G_B$ have treewidth at most two. Observe that $T_1$ and $T_2$ contain nodes $u_1$ and $u_2$ such that both the corresponding bags contain $K$. Now construct a new tree decomposition $T'$ of $G$ by creating a node $v$ with $X_v=K$ and making it adjacent to $u_1$ and $u_2$. Since $G$ is triangle-free, the size of $K$ is at most two. Therefore, the width of $T'$ is at most two.}
		\end{proof}

		\subsection{Paths, distances and layering partition}\label{sec:layer}
		
		Let $G$ be a graph. The length of a path $P$ is the number of edges in $P$. For two vertices $u,v\in V(G)$, a $(u,v)$-path (resp. $(u,v)$-induced path) is a path (resp. an induced path) between $u$ and $v$. An $(u,v)$-isometric path is an $(u,v)$-path with smallest possible length. The \emph{distance} between two vertices $u,v\in V(G)$, denoted by $\distance{G}{u}{v}$, is the length of an $(u,v)$-isometric path in $G$. For a set $S$ and a vertex $u$, the distance between $u$ and $S$, denoted by $\distance{G}{u}{S}=\min\{\distance{G}{u}{v}\colon v\in S\}$.	When $G$ is clear from the context, we use the notations $\dist{u}{v}$ and $\dist{u}{S}$.
		
		\begin{definition}[\cite{georgakopoulos2023graph}] \label{def:quasi}
			\sloppy An \emph{$(M,A)$-quasi-isometry} between graphs $G$ and $H$ is a map $f\colon V(G)\rightarrow V(H)$ such that the following holds for fixed constants $M\geq 1, A\geq 0$, \begin{enumerate*}[label=(\alph*)]
				\item $M^{-1}\dist{x}{y}-A  \leq \dist{f(x)}{f(y)} \leq M\dist{x}{y}+A$ for every $x,y\in V(G)$; and \item for every $z\in V(H)$ there is $x\in V(G)$ such that $\dist{z}{f(x)}\leq A$. 
			\end{enumerate*}
			The graphs $G$ and $H$ are \emph{quasi-isometric} if there exists a quasi-isometric map $f$ between $V(G)$ and $V(H)$. 
		\end{definition}
		
		We highlight the definition of quasi-isometry with additive distortion. 
		
		\begin{definition}\label{def:add}
			A class $\mathcal{G}$ of graphs admit a \emph{quasi-isometry with additive distortion} to a class $\mathcal{H}$ if there exists a constant $A$  (depending only on $\mathcal{G}$ and $\mathcal{H}$) such that any $G\in \mathcal{G}$ admit a $(1,A)$-quasi-isometry to a graph in $\mathcal{H}$.
		\end{definition}
		
		
		For an integer $k\geq 0, r\in V(G)$, let $\Ball{k}{r}$ denote the \emph{ball of radius $k$} centered at $r$, i.e., $\Ball{k}{r} \coloneq \{u\in G\colon \dist{r}{u}\leq k \}$. 
		For a set of vertices $S\subseteq V(G)$, define $\Ball{k}{S}\coloneq\displaystyle\bigcup\limits_{ v\in S} \Ball{k}{v}$ and \revise{$N(S)=\Ball{1}{S}\setminus S$}. 
		
		
		Let $G$ be a graph and $r$ be vertex of $G$. For $k\geq 2$ and two vertices $u,v\in \Sphere{k}{r}$, an $(u,v)$-induced path $P$ is an \emph{$(u,v)$-upper path w.r.t $r$} if $V\left(P - \{u,v\}\right)\subseteq \Ball{k-1}{r}$. Since $\Ball{k-1}{r}\cup \{u,v\}$ induces a connected subgraph of $G$, a $(u,v)$-upper path w.r.t $r$ always exists.

		
		We recall the notion of \emph{layering partition}, introduced in \cite{brandstadt1999distance, chepoi2000note}. Let $G$ be a graph and $r$ be any vertex of $G$. A subset $S \subseteq V(G)$ is a \emph{cluster w.r.t $r$} if $S$ is the maximal subset such that (i) there exists an integer $k\geq 1$ with $S\subseteq \Sphere{k}{r}$; (ii) any two vertices $x,y\in S$ lie in the same connected component in $G-\Ball{k-1}{r}$. Let $G$ be a graph, $r$ be a vertex of $G$ and $S$ be a cluster such that $S\subseteq \Sphere{k}{r}$ for some integer $k\geq 1$. The \emph{parent set} of $S$, denoted by $\ParentSet{r}{S}$, is the set $\left\{ v\in \Sphere{k-1}{r} \colon uv\in E(G), u\in S\right\}$. Observe that  $\ParentSet{r}{S}$ is a subset of some cluster w.r.t $r$.
		
		Let $G$ be a graph and $r$ be a vertex of $G$. Two clusters $C, C'$ w.r.t $r$ are  \emph{adjacent} if there is an edge in $G$ between some $x\in C$ and $y\in C'$. For a graph $G$ and a vertex $r$, let $\clusterGraphv{r}{G}$ be the graph whose vertices are the maximal clusters w.r.t $r$, and two vertices are adjacent if the corresponding clusters are adjacent. 
		
		\begin{proposition}[\cite{chepoi2000note, brandstadt1999distance}]
			For any graph $G$ and a vertex $r\in V(G)$,  $\clusterGraphv{r}{G}$ is a tree.
		\end{proposition}

		Below, we show that the parent set of every cluster is a minimal $(r,v)$-cutset where $v$ is a vertex of the cluster. 
		
		\begin{lemma}\label{lem:min-cut-set}
			Let $G$ be a graph, $r$ be a vertex, $S$ be a cluster w.r.t $r$ and $u\in S$ such that $\ParentSet{r}{S}\neq \{r\}$. Then  $\ParentSet{r}{S}$ is a minimal $(r,u)$-cutset of $G$.
		\end{lemma}
		\begin{proof}
			Let $X\coloneq \ParentSet{r}{S}$ and $S\subseteq \Sphere{k}{r}$ with $k\geq 2$.	For any $u\in S$, an $(r,s)$-path in $G$ must contain a vertex of  $X$, and therefore, $r$ and $u$ lie in different connected components in  $G-X$. Hence, $X$ is an $(r,u)$-cutset. 
			
			Let $Y\subset X$, $H=G-Y$, $u\in X\setminus Y$, and $v\in S$ with $uv\in E(H)$. Let $Z_1$ be the set of vertices that lie in $S$ or in some cluster which is a descendent of $S$ in $\clusterGraphv{r}{G}$. Observe that $Z_1\subset V(H)$ and in particular, $v\in Z_1$. Since any vertex $w\in Z_1$ is connected to $v$ by a path in  $\InducedSub{H}{Z_1}$, we have that $\InducedSub{H}{Z_1}$ is connected. Let $Z_2 = V(G-Z_1-Y)$. By definition, $V(H)=Z_1\cup Z_2, r\in Z_2$ and $u\in Z_2$. Observe that for any vertex $w\in Z_2$, there is a $(r,w)$-path in $H$. Hence $\InducedSub{H}{Z_2}$ is connected.  As $uv\in E(H)$, $H$ is connected. 
		\end{proof}
		
		\Cref{lem:min-cut-set} and \Cref{prp:ind2} imply the following properties of clusters of $K_{2,3}$-induced minor-free graphs.
		
		\begin{lemma}\label{obs:K23cut}
			Let $G$ be a $K_{2,3}$-induced minor-free graphs, $r$ be a vertex, $S$ be a cluster w.r.t $r$, and $F\coloneq \InducedSub{G}{\ParentSet{r}{S}}$. Then all of the following holds. \begin{enumerate*}[label=(\alph*)]
				\item\label{lem:twocomp} $F$ has at most two connected components,
				\item\label{it:cutk23-b} If $F$ is connected then it has diameter at most three,
				\item\label{it:cutk23-c}  If $F$ has two connected components, then each connected component is a clique.
			\end{enumerate*} 
		\end{lemma}
		
		For technical reasons, we will need the following lemma.
		\begin{lemma}\label{lem:cut-K23}
			Let $G$ be a $K_{2,3}$-induced minor-free graph, $r$ be a vertex, and $S$ be a cluster w.r.t $r$, $S_1,S_2$ be two children of $S$ in $\clusterGraphv{r}{G}$ such that for $i\in \{1,2\}$, $\InducedSub{G}{\ParentSet{r}{S_i}}$ contains two connected components $C_i$ and $D_i$. Let $\mathcal{X}=\{(Y,Z) \in \{C_1,D_1\}\times \{C_2,D_2\}\colon V(Y)\cap V(Z)\neq \emptyset\}$. Then $\dmod{\mathcal{X}}\leq 1$. 
		\end{lemma}
		\begin{proof}
			If there  exists a connected component $Y\in \{C_1,D_1\}$ containing vertices from both $C_2$ and $D_2$, then due to \Cref{obs:K23cut}\ref{it:cutk23-c}, there will be an edge in $\InducedSub{G}{\ParentSet{r}{S_2}}$ between vertices of $C_2$ and $D_2$, which is a contradiction.  Now assume $|\mathcal{X}|\geq 2$. Without loss of generality, let $C_1$ contains a vertex $c$ from $C_2$. Then, $D_1$ must contain a vertex $d$ from $D_2$. For $i\in \{1,2\}$, let $c_i\in N(c)\cap S_i$ and $d_i\in N(d)\cap S_i$. Let $k=\dist{r}{S}$. For $i\in \{1,2\}$, let $P_i$ denote a $(c_i,d_i)$-induced path in $G-\Ball{r}{k}$. (Since $c_i,d_i$ lie in the same cluster in $G$, $P_i$ exists). Let $P$ be an $(c,d)$-upper path w.r.t $r$. the graph induced by the vertices of $P,P_1,P_2$, contains a theta, which contradicts \Cref{prp:induced-truem}.
		\end{proof}

		\subsection{Strong isometric path complexity}
		
		Given a graph $G$ and a vertex $v$ of $G$, a set $S$ of isometric paths of $G$ is \emph{$v$-rooted} if $v$ is one of the end-vertices of all the isometric paths in $S$. A path $P$ is \emph{covered} by a set $\mathcal{Q}$ of paths if $V(P)\subseteq \displaystyle\bigcup\limits_{Q\in \mathcal{Q}} V(Q)$. 
		
		\begin{definition}[\cite{c2023isometric,chakraborty2025strong}]\label{def:ipco}
			The ``\emph{isometric path complexity}'' of a connected graph $G$, denoted by $\ipco{G}$, is the minimum integer $k$ such that there exists a vertex $v\in V(G)$ that satisfy the following property: any isometric path of $G$ can be covered by $k$ many $v$-rooted isometric paths. 
		\end{definition}
		
		\begin{definition}[\cite{c2023isometric,chakraborty2025strong}]\label{def:strong}
			The ``\emph{strong isometric path complexity}'' of a connected graph $G$, denoted by $\sipco{G}$, is the minimum integer $k$ such that all vertices of $v\in V(G)$ satisfy the following property: any isometric path of $G$ can be covered by $k$ many $v$-rooted isometric paths. 
		\end{definition}
		
		A result of~\cite{c2023isometric} implies that the strong isometric path complexity of graphs without a theta, pyramid, or a long prism is at most $71$. Now \Cref{prp:induced-truem}, imply the following: 
		
		\begin{proposition}\label{prp:sipco}
			The strong isometric path complexity of $K_{2,3}$-induced minor-free graphs is at most $71$.
		\end{proposition}   
		
		Below, we show how the notion of strong isometric path complexity can be used to prove quasi-isometry with additive distortion between graphs. 
		
		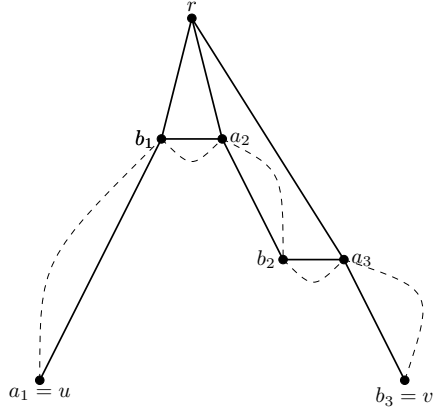
\begin{figure}
			\centering
			\scalebox{0.8}{\begin{tikzpicture}
					\foreach \x/\y [count = \n] in {5/5,3/1, 6/5, 7/3, 8/3, 9/1, 5.5/7 }
					{
						\filldraw (\x, \y) circle (2pt);
					}
					
					\foreach \x/\y/\w/\z [count = \n] in {5/5/3/1, 5/5/6/5,6/5/7/3, 7/3/8/3, 8/3/9/1, 5.5/7/5/5, 5.5/7/6/5, 5.5/7/8/3  }
					{
						\draw[thick] (\x, \y) -- (\w,\z);
					}
					
					\node[above] at (5.5,7) {$r$};
					\node [below] at (3,1) {$a_1=u$};
					\node [left] at (5,5) {$b_1$};
					\node [right] at (6,5) {$a_2$};
					\node [left] at (5,5) {$b_1$};
					\node [left] at (7,3) {$b_2$};
					\node [right] at (8,3) {$a_3$};
					\node [below] at (9,1) {$b_3=v$};
					
					\draw [dashed] (3,1) .. controls (3,3) .. (5,5); \draw [dashed] (5,5) .. controls (5.5,4.5) .. (6,5); \draw [dashed] (6,5) .. controls (7,4.5) .. (7,3);
					\draw [dashed] (7,3) .. controls (7.5,2.5) .. (8,3); \draw [dashed] (8,3) .. controls (9.5,2.5) .. (9,1); 
			\end{tikzpicture}}
			\caption{Illustration of the proof of \Cref{lem:sipco-distort}. Here $z=3$. The dashed curve from $u$ to $v$ indicates a $(u,v)$-path in $Y$. The solid black piecewise linear curve between $(u,v)$ that doesn't pass through $r$, indicates the $(u,v)$-isometric path $P$ in $X$. For $p\in \{a_1=u, b_2, b_3 =v \}$, the strictly monotone solid black curves between $r,p$ indicate a $(r,p)$-isometric path in $X$. }\label{fig:sipco-distort}
		\end{figure}
		
		\begin{lemma}\label{lem:sipco-distort}
			Let $X$ and $Y$ be two graphs such that $V(X)=V(Y)$, and $r\in V(X)$ be a vertex. Let $t\geq 1, k\geq 1$ be two integers such that the following holds. \begin{enumerate}[label=(\roman*)]
				\item for any edge $ab\in E(X)$, $\distance{Y}{a}{b}\leq k+1$;
				\item for any $r$-rooted isometric path $Q\subseteq X$ and $\{a,b\}\subseteq V(Q)$, $\distance{Y}{a}{b}\leq \distance{X}{a}{b} + t$.
			\end{enumerate} Then, for $\{u,v\}\subseteq V(X)$, $\distance{Y}{u}{v}\leq \distance{X}{u}{v} + t\cdot \sipco{X} + k(\sipco{X}-1)$.
		\end{lemma}
		\begin{proof}
			
			Let $P$ be any $(u,v)$-isometric path in $X$ and $\mathcal{Q}$ be a minimum cardinality set of $r$-rooted isometric paths in $X$ that covers $P$ and for any $Q\in \mathcal{Q}$, $|V(P)\cap V(Q)|$ is maximized. See \Cref{fig:sipco-distort} for illustration. Let $z\coloneq |\mathcal{Q}|$ and by definition, $z\leq \sipco{X}$. Observe that for any $Q\in \mathcal{Q}$, the path induced by  $V(P)\cap V(Q)$ is an isometric path in $X$.  Moreover, there exists exactly one path $Q_1\in \mathcal{Q}$ that contains $u$. Otherwise, if there exists a path $Q'\in \mathcal{Q}\setminus \{Q_1\}$ containing $u$, then either $V(P)\cap V(Q')$ contains $V(P)\cap V(Q_1)$, or vice versa.  But this contradicts the minimality of $\mathcal{Q}$.
			
			Let $Q_1\in \mathcal{Q}$ be the path containing $u$, and define $P_1\coloneq \InducedSub{P}{V(P)\cap V(Q_1)}$. 
			Let $a_1\coloneq u$, and $b_1$ be the end-vertex of $P_1$ distinct from $a_1$, \revise{unless $P_1$ is reduced to a single vertex in which case $b_1=a_1$}.  
			
			\revise{For $i\in\{2,\ldots, z\}$, let $Q_i\in \mathcal{Q}$ be a path that contains the end-vertex of the path induced by $V(P)\setminus \displaystyle\bigcup\limits_{j<i} V(Q_j)$ which is distinct from $v$. For $i\in\{2,\ldots, z\}$, let $P_i= \InducedSub{P}{V(P)\cap V(Q_1)}$, and $a_i$ is the end-vertex of $P_i$ which is adjacent to $b_{i-1}$ in $P$, and $b_i$ be the end-vertex of $P_i$ distinct from $a_i$. Note that $b_z=v$.}
			Hence,
			
			\begin{equation*} \label{eq1}
				\begin{split}
					\distance{Y}{u}{v} & = \distance{Y}{a_1}{b_z} \\
					& \leq \displaystyle\sum\limits_{i=1}^{z} \distance{Y}{a_i}{b_i} + \displaystyle\sum\limits_{i=1}^{z-1} \distance{Y}{b_i}{a_{i+1}}\\
					& \leq \displaystyle\sum\limits_{i=1}^{z} (\distance{X}{a_i}{b_i}+t) + \displaystyle\sum\limits_{i=1}^{z-1} (k+1) \hspace{10pt} \\
					& \leq   \displaystyle\sum\limits_{i=1}^{z} \distance{X}{a_i}{b_i} + \displaystyle\sum\limits_{i=1}^{z-1} \distance{X}{b_i}{a_{i+1}}+ t\cdot z + k(z-1) \\
					& \leq  \distance{X}{a_1}{b_z} + t\cdot z + k(z-1) \\
					& =   \distance{X}{u}{v} + t\cdot z + k(z-1)
				\end{split}
			\end{equation*}
			This completes the proof.
		\end{proof}

		\section{Proof of \Cref{main:th:K23}}\label{sec:proof}
		
		For the entirety of this section, $G$ shall denote a fixed $K_{2,3}$-induced minor-free graph, and $r$ shall denote a fixed vertex of $G$. All clusters (and their parent sets)  are defined w.r.t the metric in $G$ and w.r.t the vertex $r$. The output of \Cref{algo:embed} with $G$ and $r$ as input will be denoted as $H$. See \Cref{fig:example} for an example. The sets $\Sphere{k}{r}$, $\Ball{k}{r}$, $\ParentSet{r}{S}$ always contains vertices of $G$ and are defined w.r.t the metric in $G$.  In \Cref{sec:basic}, we observe some properties of $H$. In \Cref{sec:distort}, we prove lemmas relating the metric of $G$ and $H$. In \Cref{sec:cycle}, we prove some properties on the induced cycles of $H$, which we use in \Cref{sec:tree-width} to show that $H$ has tree-width at most two. We complete the proof in \Cref{sec:complete}.

		\begin{algorithm}[!h]
			\newcommand{\hrulealg}[0]{\vspace{1mm} \hrule \vspace{1mm}}
			\caption{An algorithm to embed $K_{2,3}$-induced minor-free graph on $K_4$-minor free graph.}\label{algo:embed}
			\SetKwInOut{KwIn}{Input}
			\SetKwInOut{KwOut}{Output}
			\KwIn{A $K_{2,3}$-induced minor-free graph $G$ and a vertex $r$.}
			\KwOut{A graph $H=(V(G),E')$.}
			\hrulealg

			Let $E'=\emptyset$.
			
			Construct $T=\clusterGraphv{r}{G}$. \label{line:chose-vertex}
			
			\For{each cluster $S\in V(T)$ }{
				\If{$\InducedSub{G}{ \ParentSet{r}{S}}$ is connected}{ \label{if-main}
					Choose a vertex $w$ arbitrarily in $ \ParentSet{r}{S}$. \label{line-chose-v-clique}
					
					\For{each vertex $v\in S$}{$E'=E'\cup \{vw\}$. \label{line-assign-clique-edge}}
				}
				\Else{
					Let $\InducedSub{G}{\ParentSet{r}{S}}$ has two connected components $C_1,C_2$. For $i\in \{1,2\}$, choose a vertex $w_i\in C_i$ and let $D_i\coloneq \{u\in S\colon \exists v\in C_i, uv\in E(G)\}$. \label{line-chose-v-2-clique}
					
					\If{$D_1\cap D_2=\emptyset$}{
						\For{each vertex $v\in D_1$}{$E'=E'\cup \{vw_1\}$.}\label{line:empty-parent-1}
						\For{each vertex $v\in D_2$}{$E'=E'\cup \{vw_2\}$.}\label{line:empty-parent-2}
						\If{there exists an edge $uv\in E(G)$ with $u\in D_1,v\in D_2$}
						{
							$E'=E'\cup \{uv\}$\label{line:empty}
						}
					}
					\Else{
						
						\For{each vertex $v\in D_1$}{$E'=E'\cup \{vw_1\}$.} \label{line-same-parent-1}
						\For{each vertex $v\in D_2\setminus D_1$}{$E'=E'\cup \{vw_2\}$.} \label{line-same-parent-2}
						Choose a vertex $u\in D_1\cap D_2$ arbitrarily and $E'=E'\cup \{uw_2\}$\label{line-intersection}
					}
					
				}
				
				\Return{$H=(V(G),E')$.}
			}
			
		\end{algorithm}

		\begin{figure}[t]
			\centering
			\scalebox{0.8}{
				\begin{subfigure}{0.6\textwidth}
					\begin{tikzpicture}
						\node[above] at (5,5.15) {$r$};
						\foreach \x/\y/\w/\z [count = \n] in {4.85/4.85/5.15/5.15 , 2.5/3.85/7.5/4.15 , 1.75/2.85/4.65/3.15, 4.8/2.85/5.2/3.15, 5.35/2.85/5.65/3.15, 5.85/2.85/8.15/3.15,  2.85/1.85/3.15/2.15, 3.85/1.85/4.15/2.15, 5.85/1.85/6.65/2.15, 7.35/1.85/7.65/2.15, 8.85/3.85/10.15/4.15, 8.85/2.85/10.15/3.15, 8.85/1.85/10.15/2.15}
						{
							\draw[dashed,fill=gray!40,opacity=0.5] (\x, \y) rectangle (\w,\z);
						}
						\foreach \x/\y [count = \n] in {5/5,4/4,3/4, 5/4, 6/4, 7/4, 9/4, 10/4, 2/3, 3/3, 4/3, 4.5/3, 5/3, 5.5/3, 6/3, 7/3, 7.5/3, 8/3, 9/3, 10/3, 7.5/2, 6.5/2 , 6/2, 4/2, 3/2, 10/2, 9/2 }
						{
							\filldraw (\x, \y) circle (1.5pt);
							
						}
						
						\foreach \x/\y/\w/\z [count = \n] in {5/5/4/4,5/5/4/4,5/5/3/4, 5/5/5/4, 5/5/6/4, 5/5/7/4, 3/3/2/3, 3/3/4/3,  3/4/2/3, 3/4/3/3, 3/4/4/3 ,  5/4/4.5/3,5/4/4/3, 4/3/4.5/3, 5/3/4/4, 5/4/5/3,  5.5/3/5/4,  5.5/3/6/4,  5/4/6/3, 7/4/6/3, 7/4/7/3, 7.5/3/8/3, 7/3/7.5/3, 7/4/8/3,  7/3/7.5/2, 7.5/3/7.5/2, 7.5/2/8/3 ,6.5/2/7/3 ,6.5/2/6/3 , 6/2/6/3, 6.5/2/6/2, 4.5/3/4/2, 4/3/4/2, 4/3/3/2, 2/3/3/2, 7/4/7.5/3, 5/5/9/4, 5/5/10/4, 9/4/9/3, 10/4/10/3, 9/3/9/2, 10/3/10/2, 9/2/10/2 }
						{
							\draw (\x, \y) -- (\w,\z);
						}

					\end{tikzpicture}
					\caption{A graph $G$.}
				\end{subfigure}
				\begin{subfigure}{0.6\textwidth}
					\begin{tikzpicture}
						\node[above] at (5,5.15) {$r$};
						\foreach \x/\y [count = \n] in {5/5,4/4,3/4, 5/4, 6/4, 7/4, 9/4, 10/4, 2/3, 3/3, 4/3, 4.5/3, 5/3, 5.5/3, 6/3, 7/3, 7.5/3, 8/3, 9/3, 10/3, 7.5/2, 6.5/2 , 6/2, 4/2, 3/2, 10/2, 9/2 }
						{
							\filldraw (\x, \y) circle (1.5pt);
							
						}
						
						\foreach \x/\y/\w/\z [count = \n] in {5/5/4/4,5/5/3/4, 5/5/5/4, 5/5/6/4, 5/5/7/4,5/5/9/4, 5/5/10/4, 9/4/9/3, 10/4/10/3, 9/3/9/2, 10/3/10/2, 9/2/10/2,  3/4/2/3, 3/4/3/3, 3/4/4/3 ,  5/4/4.5/3,5/4/4/3,  4/3/3/2, 2/3/3/2,  5/4/5/3,  5.5/3/5/4,  5.5/3/6/4, 5/3/4/4,  5/4/6/3, 7/4/6/3, 7/4/7.5/3, 7/4/7/3, 7/4/8/3, 4/2/4.5/3, 6/2/6/3, 6.5/2/6/3, 6.5/2/7/3, 7.5/2/7.5/3}
						{
							\draw (\x, \y) -- (\w,\z);
						}
					\end{tikzpicture}
					\caption{The graph $H$.}
			\end{subfigure}}
			\caption{ The shaded regions indicate the clusters w.r.t $r$.  The graph $H$ is the output of \Cref{algo:embed} with $G$ and $r$ as input.}\label{fig:example}
		\end{figure}
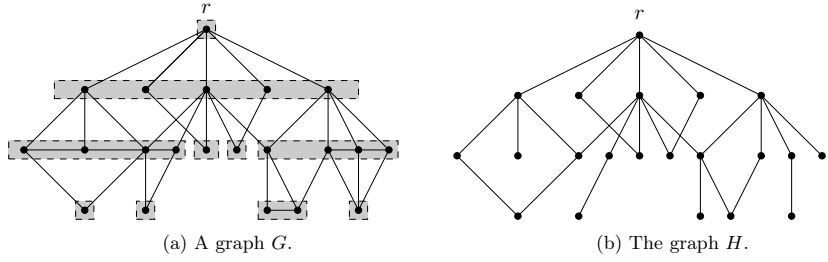

		\subsection{Basic facts about $H$}\label{sec:basic}
		
		Observe that, for every cluster $S$ of $G$ and vertex $v\in S$, we have that $v$ is adjacent to at least one vertex from $\ParentSet{r}{S}$ in $H$. Hence, the graph $H$ is connected.	Below we state some observations which follow immediately from \Cref{algo:embed}. We provide the proofs for completeness.
		
		\begin{observation}\label{obs-trivial-4}
			Let $u,v\in V(G)$ be two vertices such that $\distance{G}{r}{v}=\distance{G}{r}{u}+1$ and $uv\in E(H)$. Let $S$ be the cluster in $G$ that contains $v$. Then 
			\begin{enumerate*}[label=(\alph*),itemsep=0pt]
				\item\label{it:parent} $u\in \ParentSet{r}{S}$; and
				\item\label{it:same-comp} there exists $w\in V(G)$ such that $vw\in E(G)$ and $u,w$ lies in the same connected component in $\InducedSub{G}{\ParentSet{r}{S}}$.
			\end{enumerate*} 
		\end{observation}
		\begin{proof}
			$(a)$ follows immediately from \Cref{line-chose-v-clique,line-chose-v-2-clique} of \Cref{algo:embed}. 	To prove $(b)$, first consider the case when $F\coloneq \InducedSub{G}{\ParentSet{r}{S}}$ is connected. Clearly, $v$ has a neighbor $w\in \ParentSet{r}{S}$ such that $u$ and $w$ lie in the same connected component in $F$. Now consider the case when $F$ has two connected components $C_1,C_2$ and for $i\in \{1,2\}$, let $D_i\coloneq N(C_i)\cap S$. Since $C_1,C_2$ are cliques in $G$ (due to \Cref{obs:K23cut}\ref{it:cutk23-c}), and since there exists $i\in \{1,2\}$ such that $u\in C_i$, \Cref{algo:embed} ensures that $v\in D_i$. This implies $C_i$ has a vertex $w$ with $\{vw,uw\}\subseteq E(G)$. 
		\end{proof}
		
		\Cref{obs-trivial-4} immediately implies the following.
		
		\begin{observation}\label{obs:least-common}
			Let $P$ be an $(u,v)$-path in $H$ and $S_u,S_v$ be the clusters containing $u,v$, respectively. Let $S$ be the least common ancestor of $S_u$ and $S_v$ in $\clusterGraphv{r}{G}$. Then at least one vertex of $P$ lies in $S$. 
		\end{observation}

		\begin{observation}\label{obs-trivial-2}
			Let $S$ be a cluster  and $F\coloneq \{v\in \ParentSet{r}{S}\colon \exists u\in S, uv\in E(H)\}$. Then, \begin{enumerate*}[label=(\alph*)]
				\item\label{it:parent-deg} $|F|\leq 2$; and
				\item\label{it:parent-stable}  if $|F|=2$ then $\InducedSub{G}{\ParentSet{r}{S}}$ has two connected components in $G$.
			\end{enumerate*} 
		\end{observation}
		\begin{proof}
			The proof follows from \Cref{line-chose-v-clique,line-chose-v-2-clique} of \Cref{algo:embed}. 
		\end{proof}

		\begin{observation}\label{obs-trivial-3}
			Let $S$ be a cluster and let $F\coloneq \{u\in S\colon \exists \{v_1,v_2\}\subseteq \ParentSet{r}{S}, \{uv_1,uv_2\}\subseteq E(H)\}$. Then $|F|\leq 1$. 
		\end{observation}
		\begin{proof}
			The proof follows from \Cref{line-intersection} of \Cref{algo:embed}.
		\end{proof}
		
		\begin{observation}\label{obs-trivial-1}
			Let $S$ be a cluster and let $F \coloneq \{uv\in E(H)\colon u,v\in S\}$. Then, 	\begin{enumerate*}[label=(\alph*)]
				\item\label{it:horizontal-a} $|F|\leq 1$;
				\item \label{it:horizontal-b} if $F=\{uv\}$, then $\InducedSub{G}{\ParentSet{r}{S}}$ has two connected components, $u,v$ do not have a common neighbor in $H$, and  the set $F'=\{w\in S\colon \exists \{v_1,v_2\}\subseteq \ParentSet{r}{S}, \{wv_1,wv_2\}\subseteq E(H)\}$ is empty. 
			\end{enumerate*}
		\end{observation}
		\begin{proof}
			\revise{The proof of $(a)$ follows from \Cref{line:empty}. For $(b)$ observe that since $|F|=1$, $\InducedSub{G}{\ParentSet{r}{S}}$ has two connected components $C_1,C_2$. Moreover, the sets $N(C_1)\cap S$ and $N(C_2)\cap S$ do not intersect. Now $(b)$ follows from \Cref{line:empty-parent-1,line:empty-parent-2}.}
		\end{proof}

		We end this section by showing that the distances between $r$ and any vertex in $G$ remains the same in $H$.
		
		\begin{lemma}\label{lem:equidistant}
			For any vertex $u\in V(G)$, $\distance{H}{r}{u} = \distance{G}{r}{u}$. 
		\end{lemma}
		\begin{proof}
			Observe that $ru\in E(G)$ implies $ru\in E(H)$. Assume that for an integer $k\geq 2$, $\distance{G}{r}{u} = k-1$ implies $\distance{H}{r}{u} = \distance{G}{r}{u}$. Let $v$ be a vertex with $\distance{G}{r}{v}=k$. Let $S$ be the cluster that contains $v$.  Notice that any $(r,v)$-isometric path in $H$ contains a vertex $u\in \ParentSet{r}{S}$. Since $\ParentSet{r}{S} \subseteq \Sphere{k-1}{r}$, we have that $\distance{G}{r}{u} =\distance{H}{r}{u} =k-1$. Hence, $\distance{H}{r}{v} = k$.  
		\end{proof}

		\subsection{Distances in $G$ and $H$}\label{sec:distort}
		
		We show that, two adjacent vertices that lie in the same cluster in $G$, remain close to each other in $H$.
		
		\begin{lemma}\label{lem:edge-same-cluster}
			For any edge $xy\in E(G)$ with $\distance{G}{r}{x}=\distance{G}{r}{y}$, $\distance{H}{x}{y} \leq 5$. 
		\end{lemma}
		\begin{proof}
			Let $k=\distance{G}{r}{x}=\distance{G}{r}{y}$, and $S\subseteq \Sphere{k}{r}$ be the cluster containing $x,y$. Due to \Cref{obs:K23cut}, the graph $F\coloneq \InducedSub{G}{\ParentSet{r}{S} }$ has at most two connected components. If $F$ is connected, then consider the vertex $w$ which is chosen in \Cref{line-chose-v-clique} of \Cref{algo:embed}. Clearly, $\{xw,yw\}\subseteq E(H)$ (due to \Cref{line-assign-clique-edge} of \Cref{algo:embed}), and therefore $u,v$ has a common neighbor in $H$. Hence, $\distance{H}{x}{y}\leq 2$. Suppose $F$ has two connected components $C_1,C_2$. For $i\in\{1,2\}$, let $D_i\coloneq \{u\in S\colon \exists v\in C_i, uv\in E(G)\}$.

			Suppose $D_1\cap D_2 = \emptyset$. If there exists a $i\in \{1,2\}$, such that both $x,y\in D_i$, then there is a vertex $w\in C_i$ such that  $\{xw,yw\}\subseteq E(H)$. Otherwise, due to \Cref{line:empty}, there exists an edge $uv\in E(H)$ such that $u,x$ has a common neighbor in $H$ and $v,y$ has a common neighbor in $H$. Hence, $\distance{H}{x}{y}\leq 5$.

			Now assume $D_1\cap D_2 \neq \emptyset$.	 
			If $\{x,y\}\subseteq D_1$ or $\{x,y\}\subseteq D_2\setminus D_1$, then $x,y$ has a common neighbor in $H$ (due to \Cref{line-same-parent-1,line-same-parent-2} of \Cref{algo:embed}), and hence, $\distance{H}{x}{y}=2$.  Otherwise, without loss of generality assume $x\in D_1$ and $y\in D_2\setminus D_1$. Due to \Cref{line-intersection} of \Cref{algo:embed}, there exists a vertex $w\in D_1$ such that $x,w$ has a common neighbor in $H$ and $y,w$ has a common neighbor in $H$. Hence $\distance{H}{x}{y}\leq 4$.  
		\end{proof}

		Next, we extend \Cref{lem:edge-same-cluster} and show that adjacent vertices in $G$ remain close to each other in $H$.
		
		\begin{lemma}\label{lem:adj-distort}
			For any edge $xy\in E(G)$, $\distance{H}{x}{y}\leq 16$. 
		\end{lemma}
		\begin{proof}
			Let  $xy\in E(G)$, and without loss of generality assume that $\distance{G}{r}{y}\leq \distance{G}{r}{x}$. Let $k=\distance{G}{r}{x}$, and $S\subseteq \Sphere{k}{r}$ be the cluster containing $x$. If $\distance{G}{r}{y}=k$, then $y\in S$ and we are done by \Cref{lem:edge-same-cluster}.  Now consider the case when $\distance{G}{r}{y} = k-1$. By definition, $y\in \ParentSet{r}{S}$.  Due to \Cref{obs:K23cut}, $F\coloneq \InducedSub{G}{\ParentSet{r}{S}}$ has at most two connected components. If $F$ is connected, then let $w$ be the vertex chosen  in \Cref{line-chose-v-clique} of \Cref{algo:embed}. Note that $xw\in E(H)$. Due to \Cref{obs:K23cut}\ref{it:cutk23-b}, there is a $(y,w)$-path $P\subseteq F$ of length at most three. Applying \Cref{lem:edge-same-cluster} on the each edge of the path $P$, we have that  $\distance{H}{x}{y} \leq 16$. 
			
			Suppose $F$ has two connected components $C_1,C_2$. Due to \Cref{obs:K23cut}\ref{it:cutk23-c}, each of $C_1$ and $C_2$ are cliques in $G$. For $i\in\{1,2\}$, let $D_i\coloneq \{u\in S\colon \exists v\in C_i, uv\in E(G)\}$.  Now we have the following cases.
			
			If $D_1\cap D_2=\emptyset$, then there exists an integer $i\in \{1,2\}$ such that $x\in D_i$ and $y\in C_i$. In this case,  there exists a vertex $w\in C_i$ such that $xw\in E(H)$ (due to \Cref{line:empty-parent-1,line:empty-parent-2}). Applying  \Cref{lem:edge-same-cluster} on the edge $yw\in E(G)$, we have that $\distance{H}{x}{y}\leq 6$.
			
			Consider the case, when $D_1\cap D_2\neq \emptyset$. If $y\in C_1, x\in D_1$ or  if $y\in C_2, x\in D_2\setminus D_1$, then using analogous arguments as before we have that $\distance{H}{x}{y}\leq 6$.  Now suppose $y\in C_2$ and $x\in D_1 \cap D_2$. Then, due to \Cref{line-intersection} of \Cref{algo:embed}, there exist vertices $w_1,v,w_2$ such that $\{xw_1,w_1v,vw_2\}\subseteq E(H)$ and $w_2\in C_2$. Now applying \Cref{lem:edge-same-cluster} on the edge $yw_2\in E(G)$, we have that $\distance{H}{x}{y}\leq 8$. 
		\end{proof}
		
		Note that the graph $H$ (i.e. the output of \Cref{algo:embed}) may not be a subgraph of $G$. In the following lemma, we show that the adjacent vertices  in $H$ are also close to each other in $G$.
		\begin{lemma}\label{clm:lem-G-1}
			For any $xy\in E(H)$, $\distance{G}{x}{y}\leq 4$.
		\end{lemma}
		\begin{proof}
			Let $xy$ be any edge of $H$. Without loss of generality assume $\distance{G}{r}{y}\leq \distance{G}{r}{x}$. If $\distance{G}{r}{x} = \distance{G}{r}{y}$, then \Cref{line:empty} implies that $xy\in E(G)$.  Hence, the lemma is true in this case. Now assume $\distance{G}{r}{y} < \distance{G}{r}{x}$. \Cref{algo:embed} ensures that $\distance{G}{r}{y} = \distance{G}{r}{x}-1$. Let $S$ be the cluster that contains $x$ and let $F\coloneq \InducedSub{G}{\ParentSet{r}{S}}$. 
			Due to \Cref{obs-trivial-4}\ref{it:parent}, $y\in V(F)$. Due to \Cref{obs-trivial-4}\ref{it:same-comp}, there exists a vertex in $z\in V(F)$ such that $xz\in E(G)$ and $y,z$ lie in the same connected component in $F$. If $F$ is connected, then due to \Cref{obs:K23cut}\ref{it:cutk23-b} we have that $\distance{G}{z}{y}\leq 3$ and, thus $\distance{G}{x}{y}\leq 4$. Otherwise, due to \Cref{obs:K23cut}\ref{it:cutk23-c}, $F$ must have two connected components both of which are cliques. This implies $yz\in E(G)$ and thus $\distance{G}{x}{y}\leq 2$. 
		\end{proof}
		
		In the next lemma, we show that if a pair of vertices lie in a $r$-rooted isometric path in $G$, then their distance in $H$ is increased by an additive factor of at most $20$. 
		\begin{lemma}\label{lem:root-path-distort}
			\sloppy 	For any $r$-rooted isometric path $P\subseteq G$ and two vertices $x,y\in V(P)$, we have $\distance{H}{x}{y} \leq \distance{G}{x}{y} + 20$.
		\end{lemma}
		
		\begin{proof}
			Let $Q$ be the $(x,y)$-subpath of $P$.	Assume that the length of $Q$ is at least two. (Otherwise, we are done by \Cref{lem:adj-distort}.) Without loss of generality assume $\distance{G}{r}{y}\leq \distance{G}{r}{x}$ and $Q\coloneq x_{\ell}x_{\ell+1}\ldots,x_k$, with $x_{\ell}\coloneq y$ and $x_k\coloneq x$.  The above definitions imply $\distance{G}{x}{y}=\distance{G}{x_k}{x_\ell}=k-\ell$. Observe that any cluster contains at most one vertex of $Q$. For $i\in [\ell,k]$, let $S_i$ denote the cluster  that contains $x_i$. Let $j\in [\ell,k]$ be the minimum integer that satisfy one of the following: 
			\begin{enumerate}[label=(\roman*)]
				\item $j=k$, 
				\item $\InducedSub{G}{\ParentSet{r}{S_{j+1}}}$ is connected,
				\item \revise{$\InducedSub{G}{\ParentSet{r}{S_{j+1}}}$ has two connected components $C_1,C_2$ such that the sets  $D_1\coloneq \{u\in S_{j+1}\colon \exists v\in C_1, uv\in E(G)\} $ and   $D_2\coloneq \{u\in S_{j+1}\colon \exists v\in C_2, uv\in E(G)\}$ have a common vertex or there is an edge in $G$ between a vertex of $D_1$ and a vertex of $D_2$.}
			\end{enumerate}    
			We prove the following claim.
			\begin{claim}\label{clm:lem-dist-1}
				$\distance{H}{x_k}{x_j} \leq k-j+15$.
			\end{claim}
			\begin{subproof}
				If $k=j$, then the claim is trivially true. So we assume $j<k$. Let $x'_k\coloneq x_k$ and for each $i\in [j,k)$, let $x'_i\in S_i$ denote a vertex such that $x'_ix'_{i+1}\in E(H)$. Let $R\coloneq x'_kx'_{k-1}\ldots x'_j$. If $\InducedSub{G}{\ParentSet{r}{S_{j+1}}}$ is connected, then due to \Cref{obs:K23cut}\ref{it:cutk23-b}, there is a $(x'_j,x_j)$-induced path $Q'\subseteq F$ of length at most three in $G$. Applying  \Cref{lem:edge-same-cluster}, on each edge of $Q'$ we have that $\distance{H}{x'_j}{x_j} \leq 15$ and hence $\distance{H}{x_k}{x_j}\leq k-j+15$.

				\revise{Now suppose $\InducedSub{G}{\ParentSet{r}{S_{j+1}}}$ has two connected components $C_1,C_2$. Let 
				 $D_1\coloneq \{u\in S_{j+1}\colon \exists v\in C_1, uv\in E(G)\} $ and   $D_2\coloneq \{u\in S_{j+1}\colon \exists v\in C_2, uv\in E(G)\}$.	If there exists $i\in \{1,2\}$ with $x_j,x'_j\in C_i$, then applying \Cref{lem:edge-same-cluster} on the edge $x_jx'_j\in E(G)$, we have $\distance{H}{x'_j}{x_j}\leq 5$.
				Otherwise, without loss of generality assume $x_j\in C_1$ and $x'_j\in C_2$. If $D_1\cap D_2\neq \emptyset$, due to \Cref{line-intersection}, there exist $u\in S_{j+1}$ and $v\in S_j$ such that $\{x'_ju,uv \}\subseteq E(H)$ and $vx_j\in E(G)$. Since $\{v,x_j\}\subseteq S_j$, \Cref{lem:edge-same-cluster} implies $\distance{H}{x'_j}{x_j}\leq 7$. Hence, $\distance{H}{x_k}{x_j}\leq k-j+7$. 
				
				Otherwise, $D_1\cap D_2=\emptyset$ and there is an edge $w_1w_2\in E(G)$ such that $w_1\in D_1, w_2\in D_2$. Due to \Cref{line:empty}, there exist $u,v \in S_{j+1}$ and $v'\in S_j$ such that $\{x'_ju,uv, vv' \}\subseteq E(H)$ and $v'x_j\in E(G)$. Since $\{v',x_j\}\subseteq S_j$, \Cref{lem:edge-same-cluster} implies $\distance{H}{x'_j}{x_j}\leq 8$. Hence, $\distance{H}{x_k}{x_j}\leq k-j+8$.}
			\end{subproof}
			
			Let $x'_j=x_j$ and for each $i\in [\ell,j-1]$, let $x'_i$ be a neighbor of $x'_{i+1}$ in $H$. We prove the following claim. 
			
			\begin{claim}\label{clm:lem-dist-2}
				For each $i\in [\ell,j-1]$, $x_i=x'_i$ or $x_ix'_{i}\in E(G)$. 
			\end{claim}
			\begin{subproof}
				For each $i\in [\ell,j-1]$, let the two connected components of $\InducedSub{G}{\ParentSet{r}{S_{i+1}}}$ are $C^1_i$ and $C^2_i$. Due to \Cref{prp:ind2}, both  $C^1_i$ and $C^2_i$ are cliques in $G$. For each $t\in \{1,2\}$, let $D^t_i\coloneq \{u\in S_{i+1}\colon \exists v\in C^t_i, uv\in E(G)\}$. By definition, $D^1_i\cap D^2_i=\emptyset$ for all $i\in [\ell,j-1]$. 
				\revise{Suppose the statement of the claim is not true. Then, there exists a maximum integer $q\in [\ell,j-1]$ such that $x'_q\neq x_q$ and $x'_qx_q\not\in E(G)$. Without loss of generality, assume $x'_q\in C^1_q, x_q\in C^2_q$. Since $D^1_q\cap D^2_q=\emptyset$, $x_{q+1}\in D^2_q$ and $x'_{q+1}\in D^1_{q}$. This implies $x'_{q+1}\neq x_{q+1}$ and therefore, $q+1<j$. The definition of $q$ implies $x'_{q+1}x_{q+1}\in E(G)$. Hence, there is an edge in $G$ between a vertex of $D^1_q$ and a vertex of $D^2_q$, contradicting the minimality of $j$.  }
			\end{subproof}
			
			Let $R_1$ be a $(x_k,x_j)$-isometric path in $H$, define $R_2\coloneq x'_jx'_{j-1}\ldots x'_{\ell}$  and $R_3$ be a $(x'_\ell,x_\ell)$-isometric path in $H$. Due to \Cref{clm:lem-dist-1}, length of $R_1$ is at most $k-j+15$. Length of $R_2$ is at most $j-\ell$. Due to \Cref{clm:lem-dist-2} and \Cref{lem:edge-same-cluster}, length of $R_3$ is at most $5$. Hence, the length of the path $R_1\cup R_2\cup R_3$ is at most $k-\ell+20$. Hence $\distance{H}{x}{y} \leq \distance{G}{x}{y} + 20$.
		\end{proof}
		
		In the next lemma, we show that if a pair of vertices lie in a $r$-rooted isometric path in $H$, then their distance in $G$ is decreased by an additive factor of at most five. The proof is similar to that of \Cref{lem:root-path-distort}.
		
		\begin{lemma}\label{clm:lem-G-2}
			For any $r$-rooted isometric path $P\subseteq H$, and two vertices $x,y\in V(P)$, $\distance{G}{u}{v} - 6 \leq \distance{H}{u}{v}$.
		\end{lemma}
		\begin{proof}
			Without loss of generality, assume that $\distance{H}{r}{y} \leq \distance{H}{r}{x}$.  If $xy\in E(H)$, then we are done by \Cref{clm:lem-G-1}. Otherwise, let $Q$ be the $(x,y)$-subpath of $P$. Let $k\coloneq \distance{H}{r}{x}$, $\ell\coloneq \distance{H}{r}{y}$, $x_k \coloneq x$, $x_\ell=y$ and $Q\coloneq x_kx_{k-1}\ldots x_\ell$. Observe that for $i\in [\ell,k]$, $\distance{H}{r}{x_i}=i$. Due to \Cref{lem:equidistant}, $\distance{G}{r}{x_i}=i$. Therefore, any cluster contains at most one vertex of $Q$. For $i\in [\ell,k]$, let $S_i$ be the cluster that contains $x_i$. 
			Due to \Cref{obs-trivial-4}\ref{it:parent}, for $i\in [\ell+1,k]$, $x_{i-1}\in \ParentSet{r}{S_{i}}$. Let $j\in [\ell,k]$ be the minimum integer that satisfy the following: 
			
			\begin{enumerate}[label=(\roman*)]
				\item $j=k$, 
				\item $\InducedSub{G}{\ParentSet{r}{S_{j+1}}}$ is connected,
				\item \revise{$\InducedSub{G}{\ParentSet{r}{S_{j+1}}}$ has two connected components $C_1,C_2$ such that the sets  $D_1\coloneq \{u\in S_{j+1}\colon \exists v\in C_1, uv\in E(G)\} $ and   $D_2\coloneq \{u\in S_{j+1}\colon \exists v\in C_2, uv\in E(G)\}$ have a common vertex or there is an edge in $G$ between a vertex of $D_1$ and a vertex of $D_2$.}
			\end{enumerate}    
			We prove the following claim.
			\begin{claim}\label{clm:lem-G-2-1}
				$\distance{G}{x_k}{x_j} \leq k-j+5$.
			\end{claim}
			\begin{subproof}
				If $k=j$, then the claim is trivially true. So we assume $j<k$. Let $x'_k\coloneq x_k$ and for each $i\in [k-1,j]$, let $x'_i\in S_i$ denote a vertex such that $x'_ix'_{i+1}\in E(G)$. Let $R\coloneq x'_kx'_{k-1}\ldots x'_j$. If $\InducedSub{G}{\ParentSet{r}{S_{j+1}}}$ is connected, then due to \Cref{obs:K23cut}\ref{it:cutk23-b}, there is a $(x'_j,x_j)$-induced path $Q'\subseteq F$ of length at most three in $G$. Hence, $\distance{G}{x_k}{x_j}\leq k-j+4$. 
				
				\revise{Now suppose $\InducedSub{G}{\ParentSet{r}{S_{j+1}}}$ has two connected components $C_1,C_2$. Let 
				$D_1\coloneq \{u\in S_{j+1}\colon \exists v\in C_1, uv\in E(G)\} $ and   $D_2\coloneq \{u\in S_{j+1}\colon \exists v\in C_2, uv\in E(G)\}$. If $x'_jx_j\in E(G)$, then $d_G(x_k,x_j)\leq k-j+1$. Otherwise, without loss of generality assume $x'_j\in C_1$ and $x_j\in C_2$. If $D_1\cap D_2\neq \emptyset$, there exist $u\in S_{j+1}$ and $w,v\in S_j$ such that $\{x'_jw,wu,uv,vx_j \}\subseteq E(G)$. Hence, $\distance{G}{x'_j}{x_j}\leq 4$ and thus, $\distance{G}{x_k}{x_j}\leq k-j+4$. 
				
				Otherwise, there must exist an edge in $G$ between a vertex of $D_1$ and a vertex of $D_2$. In this case, 
				there exist $u',v'\in S_{j+1}$ and $w,v\in S_j$ such that $\{x'_jw,wu',u'v',v'v, vx_j \}\subseteq E(G)$. Hence, $\distance{G}{x'_j}{x_j}\leq 5$ and thus, $\distance{G}{x_k}{x_j}\leq k-j+5$.}			
			\end{subproof}
			
			Let $x'_j=x_j$ and for each $i\in [\ell,j-1]$, let $x'_i\in \ParentSet{r}{S_{i+1}}$ be a vertex such that $x'_ix'_{i+1}\in E(G)$. We prove the following claim. 
			
			\begin{claim}\label{clm:lem-G-2-2}
				For each $i\in [\ell,j-1]$, $x_i=x'_i$ or $x_ix'_{i}\in E(G)$. 
			\end{claim}
			\begin{subproof}
				For each $i\in [\ell,j-1]$, let the two connected components of $\InducedSub{G}{\ParentSet{r}{S_{i+1}}}$ are $C^1_i$ and $C^2_i$. Due to \Cref{prp:ind2}, both  $C^1_i$ and $C^2_i$ are cliques in $G$. For each $t\in \{1,2\}$, let $D^t_i\coloneq \{u\in S_{i+1}\colon \exists v\in C^t_i, uv\in E(G)\}$. By definition, $D^1_i\cap D^2_i=\emptyset$. 
				\revise{Suppose the statement of the claim is not true. Then, there exists a maximum integer $q\in [\ell,j-1]$ such that $x'_q\neq x_q$ and $x'_qx_q\not\in E(G)$. Without loss of generality, assume $x'_q\in C^1_q, x_q\in C^2_q$. Since $D^1_q\cap D^2_q=\emptyset$, $x_{q+1}\in D^2_q$ and $x'_{q+1}\in D^1_{q}$. This implies $x'_{q+1}\neq x_{q+1}$ and therefore, $q+1<j$. The definition of $q$ implies $x'_{q+1}x_{q+1}\in E(G)$. Hence, there is an edge in $G$ between a vertex of $D^1_q$ and a vertex of $D^2_q$, contradicting the minimality of $j$.  }
			\end{subproof}
			
			Let $R_1$ be a $(x_k,x_j)$-isometric path in $G$, and $R_2\coloneq x'_jx'_{j-1}\ldots x'_{\ell}$. Due to \Cref{clm:lem-G-2-1}, length of $R_1$ is at most $k-j+5$. Length of $R_2$ is $j-\ell$. Due to \Cref{clm:lem-G-2-2}, $x_\ell x'_\ell \in E(G)$. Hence $\distance{G}{x}{y} \leq \distance{H}{x}{y} + 6$.
		\end{proof}
		\subsection{Induced cycles in $H$}\label{sec:cycle}
		
		The main purpose of this section is to prove that any cluster of $G$ contains at most two vertices of an induced cycle in $H$ (see \Cref{lem:cycle}).
		We start by proving the following.
		
		\begin{lemma}\label{lem:same-cluster}
			Let $C$ be an induced cycle in $H$. For $k\geq 1$, all vertices in $V(C)\cap \Sphere{k}{r}$ lies in the same cluster. 
		\end{lemma}
		\begin{proof}
			Assume there exists an integer $k$ and two vertices $u_1,u_2\in V(C)\cap \Sphere{k}{r}$ such that $u_1,u_2$ lie in different clusters $S_1,S_2$ respectively. Let $S$ be the least common ancestor of $S_1$ and $S_2$ in $\clusterGraphv{r}{G}$. Observe that $S,S_1,S_2$ are all distinct clusters. For $i\in \{1,2\}$, let $S'_i$ denote the cluster which is a child of $S$ and lie in the $(S,S_i)$-path in $\clusterGraphv{r}{G}$. Observe that $S,S'_1,S'_2$ are distinct clusters. Since there are two $(u_1,u_2)$-paths in $C$ (and thus in $H$) which are internally vertex disjoint, due to \Cref{obs:least-common}, $S$ must contain two vertices $v_1,v_2\in V(C)$. Observe that for $i\in\{1,2\},j\in\{1,2\}$, $v_i$ must have a neighbor in $C$ (and thus in $H$) that lies in $S'_j$. Now applying \Cref{obs-trivial-2}\ref{it:parent-stable} on $S'_1$, we have that $\InducedSub{G}{\ParentSet{r}{S'_1}}$ has two connected components $C_1,D_1$. Similarly, applying \Cref{obs-trivial-2}\ref{it:parent-stable} on $S'_2$, we have that $\InducedSub{G}{\ParentSet{r}{S'_2}}$ has two connected components $C_2,D_2$. Without loss of generality, $v_1\in C_1\cap C_2$ and $v_2\in D_1\cap D_2$. But this contradicts \Cref{lem:cut-K23}.
		\end{proof}
		
		\begin{lemma}\label{lem:cycle-end}
			Let $C$ be an induced cycle in $H$, $k=\max\{\distance{H}{r}{v}\colon v\in V(C)\}$ and $\ell=\distance{H}{r}{V(C)}$. Then,  $(a)~|\Sphere{k}{r}\cap V(C)|\leq 2$, and $(b)~|\Sphere{\ell}{r} \cap V(C)|\leq 2$. 
		\end{lemma}
		\begin{proof}
			Due to \Cref{lem:equidistant}, $k=\max\{\distance{G}{r}{v}\colon v\in V(C)\}$ and $\ell=\distance{G}{r}{V(C)}$. Let $S_k\subseteq \Sphere{k}{r}$ and $S_\ell\subseteq \Sphere{\ell}{r}$ denote the clusters that contain some vertices of $C$.  Due to \Cref{lem:same-cluster}, $V(C)\cap \Sphere{k}{r}\subseteq S_k$ and $V(C)\cap \Sphere{\ell}{r}\subseteq S_\ell$. Now,  $(a)$ follows from \Cref{obs-trivial-3} and \Cref{obs-trivial-1}. 
			
			To prove $(b)$, assume $\{a,b,c\}\subseteq S_\ell\cap V(C)$ and $F=\{uv\in E(H)\colon \{u,v\}\subseteq \{a,b,c\}\}$. By \Cref{obs-trivial-1}\ref{it:horizontal-a}, $|F|\leq 1$. Due to \Cref{lem:same-cluster}, there exists a child $S'$ of $S_\ell$ in $\clusterGraphv{r}{G}$ such that for each $d\in \{a,b,c\}$ there exists $d'\in S'$ such that $dd'\in E(H)$. This contradicts \Cref{obs-trivial-2}\ref{it:parent-deg}.
		\end{proof}
		
		We also observe that $H$ is triangle-free. 
		
		\begin{lemma}\label{lem:triangle}
			The graph $H$ is triangle-free.
		\end{lemma}
		\begin{proof}
			Assume for contradiction that, $H$ contains a triangle $T$. Then there must exist a cluster $S$ w.r.t $r$ and an edge $uv\in E(T)$ such that $\{u,v\}\subseteq S$. This contradicts  \Cref{obs-trivial-1}\ref{it:horizontal-b}. 
		\end{proof}

		Now we are ready to prove the main lemma of this section.
		
		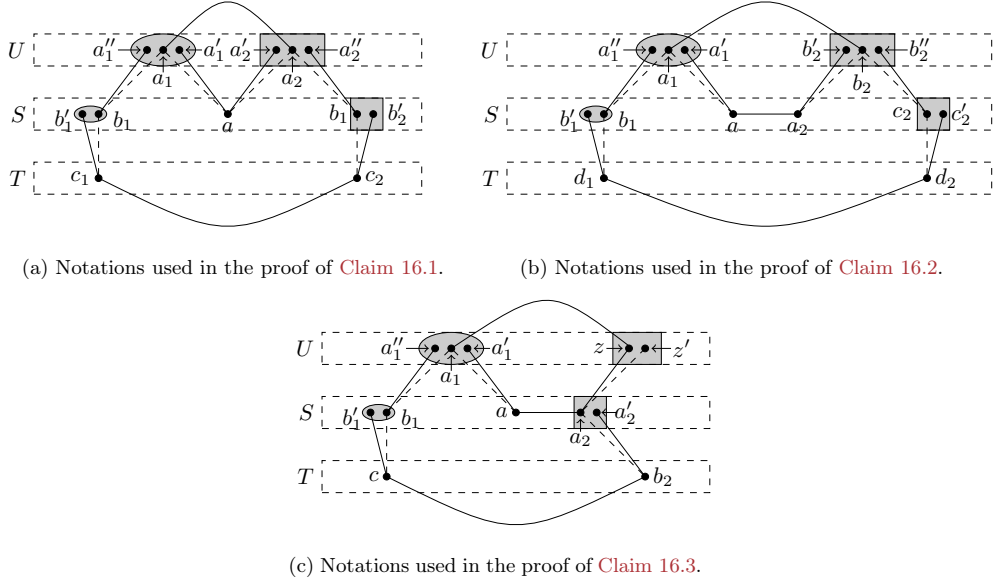
\begin{figure}
			\centering
			\scalebox{0.85}{
				\begin{subfigure}{0.45\textwidth}
					\begin{tikzpicture}
						\filldraw[fill=gray!40,opacity=0.5] (-1,1) ellipse (0.5cm and 0.25cm);
						\filldraw[fill=gray!40,opacity=0.5] (0.5,0.75) rectangle (1.5,1.25);
						\filldraw[fill=gray!40,opacity=0.5] (-2.125,0) ellipse (0.25cm and 0.125cm);
						\filldraw[fill=gray!40,opacity=0.5] (1.9,-0.25) rectangle (2.4,0.25);
						
						\foreach \x/\y [count = \n] in {0/0,-1/1,1/1, 2/0, -2/0, -2/-1, 2/-1,1.25/1, -1.25/1, 1.25/1, -0.75/1,0.75/1, -2.25/0,2.25/0 }
						{
							\filldraw (\x, \y) circle (1.5pt);
							
						}
						\foreach \x/\y/\w/\z [count = \n] in {0/0/-1/1,1/1/0/0, 2/0/1/1, -2/0/-1/1, -2/0/-2/-1, 2/0/2/-1}
						{
							\draw[dashed] (\x, \y) -- (\w,\z);
							
						}
						
						\foreach \x/\y/\w/\z [count = \n] in {0/0/-0.75/1,0.75/1/0/0, 2/0/1.25/1, -2/0/-1.25/1, -2.25/0/-2/-1, 2.25/0/2/-1}
						{
							\draw (\x, \y) -- (\w,\z);
							
						}
						\draw (-2,-1)  ..controls (0,-2) ..  (2,-1);
						\draw (-1,1) ..controls (0,2) .. (1,1);
						
						\node[below] at (0,0) {$a$};
						\node[below] at (-1,0.75) {$a_1$}; \draw [->] (-1,0.65) -- (-1,0.9);
						\node[right] at (-0.5,1) {$a'_1$}; \draw [<-] (-0.65,1) -- (-0.4,1);
						\node[left] at (-1.6,1) {$a''_1$}; \draw [->] (-1.7,1) -- (-1.35,1);
						
						\node[below] at (1,0.75) {$a_2$}; \draw [->] (1,0.65) -- (1,0.9);
						\node[left] at (0.5,1) {$a'_2$}; \draw [<-] (0.65,1) -- (0.4,1);
						\node[right] at (1.6,1) {$a''_2$}; \draw [->] (1.7,1) -- (1.35,1);
						
						\node[right] at (-1.9,-0.1) {$b_1$}; \node[left] at (-2.25,-0.1) {$b'_1$}; 
						
						\node[right] at (2.35,0) {$b'_2$}; \node[left] at (2,0) {$b_1$}; 
						
						\node[left] at (-2,-1) {$c_1$}; \node[right] at (2,-1) {$c_2$}; 
						
						\draw[dashed] (-3,0.75) rectangle (3,1.25);
						\draw[dashed] (-3,-0.25) rectangle (3,0.25);
						\draw[dashed] (-3,-1.25) rectangle (3,-0.75);
						\node[left] at (-3,1) {$U$}; 
						\node[left] at (-3,0) {$S$}; 
						\node[left] at (-3,-1) {$T$}; 
					\end{tikzpicture}
					\caption{Notations used  in the proof of  \Cref{clm:9.2}.}\label{subfig:case-1}
				\end{subfigure} 
				\begin{subfigure}{0.5\textwidth}
					\begin{tikzpicture}
						\filldraw[fill=gray!40,opacity=0.5] (-1,1) ellipse (0.5cm and 0.25cm);
						\filldraw[fill=gray!40,opacity=0.5] (1.5,0.75) rectangle (2.5,1.25);
						\filldraw[fill=gray!40,opacity=0.5] (-2.125,0) ellipse (0.25cm and 0.125cm);
						\filldraw[fill=gray!40,opacity=0.5] (2.85,-0.25) rectangle (3.35,0.25); 
						\foreach \x/\y [count = \n] in {0/0,-1/1, -1.25/1, -0.75/1,  1/0, 2.25/1, 1.75/1, 2/1, -2/0, -2.25/0, 3/0,  -2/-1, 3/-1, 3.25/0 }
						{
							\filldraw (\x, \y) circle (1.5pt);
							
						}
						\foreach \x/\y/\w/\z [count = \n] in {0/0/-1/1,2/1/1/0, 3/0/2/1, -2/0/-1/1, -2/0/-2/-1, 3/0/3/-1}
						{
							\draw[dashed] (\x, \y) -- (\w,\z);
							
						}
						\foreach \x/\y/\w/\z [count = \n] in {0/0/-0.75/1, 1.75/1/1/0, 0/0/1/0, -2/0/-1.25/1, 3/0/2.25/1, -2.25/0/-2/-1, 3.25/0/3/-1}
						{
							\draw (\x, \y) -- (\w,\z);
							
						}
						\draw (-2,-1)  ..controls (0.5,-2) ..  (3,-1);
						\draw (-1,1) ..controls (0.5,2) .. (2,1);
						
						\node[below] at (0,0) {$a$};\node[below] at (1,0) {$a_2$};
						\node[below] at (-1,0.75) {$a_1$}; \draw [->] (-1,0.65) -- (-1,0.9);
						\node[right] at (-0.5,1) {$a'_1$}; \draw [<-] (-0.65,1) -- (-0.4,1);
						\node[left] at (-1.6,1) {$a''_1$}; \draw [->] (-1.7,1) -- (-1.35,1);
						
						\node[right] at (-1.9,-0.1) {$b_1$}; \node[left] at (-2.25,-0.1) {$b'_1$}; 
						
						\node[right] at (3.25,0) {$c'_2$}; \node[left] at (2.9,0) {$c_2$}; 
						
						\node[below] at (2,0.75) {$b_2$}; \draw [->] (2,0.65) -- (2,0.9);
						\node[left] at (1.5,1) {$b'_2$}; \draw [<-] (1.65,1) -- (1.4,1);
						\node[right] at (2.6,1) {$b''_2$}; \draw [->] (2.7,1) -- (2.35,1);
						
						\node[left] at (-2,-1) {$d_1$}; \node[right] at (3,-1) {$d_2$}; 
						\node[left] at (-3.5,1) {$U$}; 
						\node[left] at (-3.5,0) {$S$}; 
						\node[left] at (-3.5,-1) {$T$}; 
						\draw[dashed] (-3.5,0.75) rectangle (4,1.25);
						\draw[dashed] (-3.5,-0.25) rectangle (4,0.25);
						\draw[dashed] (-3.5,-1.25) rectangle (4,-0.75);
					\end{tikzpicture}
					\caption{Notations used in the proof of \Cref{clm:9.3}.}\label{subfig:case-2}
			\end{subfigure} }
			
			\scalebox{0.85}{
				\begin{subfigure}{0.4\textwidth}
					\begin{tikzpicture}
						\filldraw[fill=gray!40,opacity=0.5] (-1,1) ellipse (0.5cm and 0.25cm);
						\filldraw[fill=gray!40,opacity=0.5] (1.5,0.75) rectangle (2.25,1.25);
						\filldraw[fill=gray!40,opacity=0.5] (-2.125,0) ellipse (0.25cm and 0.125cm);
						\filldraw[fill=gray!40,opacity=0.5] (0.9,-0.25) rectangle (1.4,0.25); 
						
						\foreach \x/\y [count = \n] in {0/0,-1/1, -1.25/1, -0.75/1,  1/0, 1.25/0, 1.75/1, 2/1, -2/0, -2.25/0,  -2/-1, 2/-1 }
						{
							\filldraw (\x, \y) circle (1.5pt);
							
						}
						\foreach \x/\y/\w/\z [count = \n] in {0/0/-1/1,2/1/1/0, -2/0/-1/1, -2/0/-2/-1, 1/0/2/-1}
						{
							\draw[dashed] (\x, \y) -- (\w,\z);
							
						}
						\foreach \x/\y/\w/\z [count = \n] in {0/0/-0.75/1, 1.75/1/1/0, 0/0/1/0, -2/0/-1.25/1, -2.25/0/-2/-1, 1.25/0/2/-1}
						{
							\draw (\x, \y) -- (\w,\z);
							
						}
						\draw (-2,-1)  ..controls (0,-2) ..  (2,-1);
						\draw (-1,1) ..controls (0.5,2) .. (1.8,1);
						\node[left] at (0,0) {$a$};
						\node[below] at (-1,0.75) {$a_1$}; \draw [->] (-1,0.65) -- (-1,0.9);
						\node[right] at (-0.5,1) {$a'_1$}; \draw [<-] (-0.65,1) -- (-0.4,1);
						\node[left] at (-1.6,1) {$a''_1$}; \draw [->] (-1.7,1) -- (-1.35,1);
						\node[right] at (-1.9,-0.1) {$b_1$}; \node[left] at (-2.25,-0.1) {$b'_1$}; 
						
						\node[below] at (1,-0.2) {$a_2$};  \draw [->] (1,-0.3) -- (1,-0.1);	\node[right] at (1.4,0) {$a'_2$};  \draw [->] (1.55,0) -- (1.325,0);
						\node[left] at (1.5,1) {$z$}; \draw [<-] (1.65,1) -- (1.4,1);
						\node[right] at (2.3,1) {$z'$}; \draw [->] (2.4,1) -- (2.1,1);
						
						\node[left] at (-2,-1) {$c$}; \node[right] at (2,-1) {$b_2$}; 
						
						\node[left] at (-3,1) {$U$}; 
						\node[left] at (-3,0) {$S$}; 
						\node[left] at (-3,-1) {$T$}; 
						
						\draw[dashed] (-3,0.75) rectangle (3,1.25);
						\draw[dashed] (-3,-0.25) rectangle (3,0.25);
						\draw[dashed] (-3,-1.25) rectangle (3,-0.75);

					\end{tikzpicture}
					\caption{Notations used in the proof of \Cref{clm:9.4}.}\label{subfig:case-3}
			\end{subfigure} }

			\caption{Illustration of the proof of \Cref{lem:cycle}. The dashed lines indicate edges in $H$. The solid lines indicate edges in $G$. The  curved lines indicate induced paths in $G$. The vertices in the shaded region indicate cliques in $G$. There are no edges in $G$ between a vertex in a rectangular region and a vertex in a elliptical region.}
		\end{figure}
		\begin{lemma}\label{lem:cycle}
			Let $C$ be an induced cycle in $H$ and $S$ be a cluster. Then $|S\cap V(C)|\leq 2$. 
		\end{lemma}
		\begin{proof}
			Let $k=\max\{\distance{H}{r}{v}\colon v\in V(C)\}$ and $\ell=\distance{H}{r}{V(C)}$. Due to \Cref{lem:equidistant}, $k=\max\{\distance{G}{r}{v}\colon v\in V(C)\}$ and $\ell=\distance{G}{r}{V(C)}$. Due to \Cref{lem:same-cluster}, there are clusters $S_k$ and $S_\ell$ such that $V(C)\cap \Sphere{k}{r}\subseteq S_k$ and  $V(C)\cap \Sphere{k}{r}\subseteq S_\ell$. Due to \Cref{lem:cycle-end}, we have that $|S_k\cap V(C)|\leq 2$ and $|S_\ell \cap V(C)| \leq 2$. Moreover, all vertices of $C$ lie in the clusters that lie in the unique $(S_k,S_\ell)$-path in $\clusterGraphv{r}{G}$.
			
			Assume the lemma is not true, and let $k'$ be the maximum integer such that there exists a cluster $S\subseteq \Sphere{k'}{r}$ with $|V(C)\cap S|\geq 3$. 
			Due to \Cref{lem:cycle-end}, $\ell <k'<k$. 
			Let $T$ be the child of $S$ that lie in the unique $(S_k,S)$-path in 
			$\clusterGraphv{r}{G}$, and $U$ be the parent of $S$  in $\clusterGraphv{r}{G}$. Observe that $U\subseteq \Sphere{k'-1}{r}$ and $T\subseteq \Sphere{k'+1}{r}$. Due to \Cref{obs-trivial-2}\ref{it:parent-deg}, there exists at least one vertex $a\in V(C)\cap S$ such that $T$ does not contain any vertex of $C$ which is adjacent to $a$. Let $a_1,a_2$ be the two neighbors of $a$ in $C$. We prove the following claim.
			
			%
			\begin{claim}\label{clm:9.2}
				$\{a_1,a_2\}\cap S\neq \emptyset$. 	
			\end{claim}
			\begin{subproof}
				Assuming otherwise, let $\{a_1,a_2\}\subseteq U$. See \Cref{subfig:case-1} for notations. Since $U$ is the parent of $S$ in $\clusterGraphv{r}{G}$, $\ParentSet{r}{S}\subseteq U$.  Now, $\ParentSet{r}{S}$ contains two vertices $a_1,a_2$ which is adjacent to some vertices (e.g. $a$) in $S$. Due to \Cref{obs-trivial-2}\ref{it:parent-stable}, we have that $\InducedSub{G}{\ParentSet{r}{S}}$ have two connected components $C_1,C_2$ such that $a_1\in C_1, a_2\in C_2$. Observe that, $C_1,C_2$ induce cliques in $G$ and that no vertex of $U\setminus \{a_1,a_2\}$ is adjacent to any vertex of $S$.
				
				Let $x$ be a vertex in $S_k\cap V(C)$.  Observe that there exist a $(x,a_1)$-path $P_1\subseteq C$ that does not contain $a$. Similarly, there exists a $(x,a_2)$-path  $P_2\subseteq C$ that does not contain $a$. Observe that $P_1$ and $P_2$ only meets at $x$. Moreover, for $i\in \{1,2\}$ there must exists a vertex $b_i\in S\cap V(P_i)$ such that $a_ib_i\in E(C)$. For $i\in \{1,2\}$, let $c_i\in V(P_i)\setminus \{a_i\}$ be the neighbor of $b_i$ distinct from $a_i$. For each $i\in \{1,2\}$, observe that $c_i\in T$, otherwise $c_i$ will be adjacent to $a_1$ or $a_2$ (in $H$) contradicting the fact that $C$ is an induced cycle.
				
				Now apply \Cref{obs-trivial-4} on each of the edges in $\{ aa_1,aa_2, b_1a_1,b_2a_2\}\subseteq E(H)$ to infer that there exists vertices $\{a_1',a_2',a''_1,a''_2\}\subseteq U$ such that $\{aa'_1,aa'_2,b_1a''_1,b_2a''_2\}\subseteq E(G)$ and $\{a'_1,a''_1,a_1\}\subseteq C_1$ and $\{a'_2,a''_2,a_2\}\subseteq C_2$. Similarly, apply \Cref{obs-trivial-4} on each of the edges in $\{ b_1c_1,b_2c_2\}\subseteq E(H)$ to infer that there exists vertices $\{b'_1,b'_2\}\subseteq S$ such that $\{b'_1c_1,b'_2c_2, b_1b'_1, b_2b'_2\}\subseteq E(G)$. 
				
				Since $c_1,c_2$ lies in the same cluster $T\subseteq \Sphere{k+1}{r}$, there exists a $(c_1,c_2)$-induced path $Q$ in $G-\Ball{k}{r}$.  On the other hand there exists a $(a_1,a_2)$-upper path (see \Cref{sec:layer} for definition) $R$ in $G$ such that $(V(R)\setminus \{a_1,a_2\})\subseteq \Ball{k-2}{r}$. Now consider the vertices $$Z\coloneq V(R)\cup \{a,a_1,a_2,a'_1,a'_2,a''_1,a''_2\}\cup \{b_1,b_2,b'_1,b'_2\}\cup \{c_1,c_2\}\cup V(Q)$$ It can be verified that  that $\InducedSub{G}{Z}$ contains a theta, long prism, pyramid, or a broken wheel as an induced subgraph, which contradicts \Cref{prp:induced-truem}. 
			\end{subproof}

			Due to	\Cref{clm:9.2}, we assume (without loss of generality) $a_1\in U, a_2\in S$. 
			For $i\in \{1,2\}$, let $b_i$ be the neighbor of $a_i$ in $C$ which is distinct from $a$. Now we prove the following claim. 
			
			\begin{claim}\label{clm:9.3}
				The vertex $b_2$ does not lie in $U$.
			\end{claim}
			\begin{subproof}
				Assume $b_2\in U$. See \Cref{subfig:case-2} for notations. Due to \Cref{obs-trivial-2}, we have that $\InducedSub{G}{\ParentSet{r}{S}}$ have two connected components $A,B$ such that $a_1\in A, b_2\in B$. Moreover, $A,B$ induce cliques in $G$ and no vertex of $U\setminus \{a_1,b_2\}$ is adjacent to any vertex of $S$.
				
				Now consider a vertex $x\in V(C)\cap S_k$.   Observe that there exist a $(x,a_1)$-path $P_1\subseteq C$ that does not contain $a$ or $a_2$. Similarly, there exists a $(x,b_2)$-path  $P_2\subseteq C$ that does not contain $a,a_1$ or $a_2$.  Moreovoer, $V(P_1)\cap V(P_2)=\{x\}$ and the edge $a_1b_1$ lies in  $E(P_1)$.  Let $c_2$ be the neighbor of $b_2$ in $P_2$. Due to \Cref{obs-trivial-2}, we have that $b_1\in S$ and $c_2\in S$. 
				
				Let $d_1\in V(P_1)$ be the neighbor of $b_1$ distinct from $a_1$. Let $d_2\in V(P_2)$ be the neighbor of $c_2$ distinct from $b_2$. For each $i\in \{1,2\}$, observe that $d_i\in T$, otherwise $d_i$ will be adjacent to $a_1$ or $b_2$ contradicting the fact that $C$ is an induced cycle. Now apply \Cref{obs-trivial-4} on each of the edges in $\{ aa_1,a_2b_2, b_1a_1,b_2c_2\}\subseteq E(H)$ to infer that there exists vertices $\{a_1',b_2',a''_1,b''_2\}\subseteq U$ such that $\{aa'_1,a_2a'_2,b_1a''_1,c_2b''_2\}\subseteq E(G)$, $\{a'_1,a''_1,a_1\}\subseteq A$ and $\{b'_2,b''_2,b_2\}\subseteq B$.	Similarly, apply \Cref{obs-trivial-4} on  edges in $\{ b_1d_1,c_2d_2\}\subseteq E(H)$ to infer that there exists vertices $\{b'_1,c'_2\}\subseteq S$ such that $\{b'_1d_1,d_2c'_2,b_1,b'_1,c_2c'_2\}\subseteq E(G)$. 
				
				Since $d_1,d_2$ lies in the same cluster $T\subseteq \Sphere{k+1}{r}$, there exists a $(d_1,d_2)$-induced path $Q$ in $G-\Ball{k}{r}$. On the other hand there exists a $(a_1,b_2)$-upper path $R$ such that $(V(R)\setminus \{a_1,a_2\})\subseteq \Ball{k-2}{r}$.   Now consider the vertices $$Z\coloneq V(R)\cup \{a,a_1,a_2,a'_1,b'_2,a''_1,b''_2\}\cup \{b_1,b_2,b'_1,b'_2\}\cup \{c_1,c_2\}\cup V(Q)$$ It can be verified that $\InducedSub{G}{Z}$ contains a theta, long prism, pyramid, or a broken wheel as an induced subgraph, which contradicts \Cref{prp:induced-truem}. 
			\end{subproof}  
			
			We also prove the following claim.

			%
			\begin{claim}\label{clm:9.4}
				The vertex $b_2$ does not lie in $T$.
			\end{claim}
			\begin{subproof}
				Assume $b_2\in T$. See \Cref{subfig:case-3} for notations. Let $z$ be the vertex in $U$ which is also adjacent to $a_2$ in $H$. Due to \Cref{obs-trivial-1}, we have that $z$ is distinct from $a_1$. Moreover, due to \Cref{obs-trivial-2}, we have that $\InducedSub{G}{\ParentSet{r}{S}}$ have two connected components $A,B$ such that $a_1\in A, z\in B$. Moreover, $A,B$ induce cliques in $G$ and no vertex of $U\setminus \{a_1,z\}$ is adjacent to any vertex of $S$.
				
				Now consider a vertex $x\in V(C)\cap S_k$.   Observe that there exists a $(x,a_1)$-path $P\subseteq C$ that does not contain $a$ or $a_2$. Moreover, $a_1b_1\in E(P)$ and $z\notin V(P)$. Due to \Cref{obs-trivial-2}, we have that $b_1\in S$. Let $c\in V(P)$ be the neighbor of $b_1$ distinct from $a_1$. Since $H$ is triangle-free, $c$ is not adjacent to $a_1$. Due to \Cref{obs-trivial-1}, we have $c\neq z$ and moreover, $c$ is not adjacent to $z$ in $H$. Hence, $c\notin U, c\notin S$. Therefore, $c$ must lie in  $T$. 
				
				Now apply \Cref{obs-trivial-4} on each of the edges in $\{ aa_1,a_2z, b_1a_1\}\subseteq E(H)$ to infer that there exists vertices $\{a_1',z',a''_1\}\subseteq U$ such that $\{aa'_1,a_2z',b_1a''_1\}\subseteq E(G)$, $\{a'_1,a''_1,a_1\}\subseteq A$ and $\{z,z'\}\subseteq B$.	Similarly, apply \Cref{obs-trivial-4} on  edges in $\{ b_1c,a_2b_2\}\subseteq E(H)$ to infer that there exists vertices $\{b'_1,a'_2\}\subseteq S$ such that $\{b'_1c, a'_2b_2,b_1b'_1,a_2a'_2\}\subseteq E(G)$. 
				
				Since $c,b_2$ lies in the same cluster $T\subseteq \Sphere{k+1}{r}$, there exists a $(c,b_2)$-induced path $Q$ in $G-\Ball{k}{r}$. On the other hand there exists a $(a_1,z)$-upper path $R$ such that $(V(R)\setminus \{a_1,z\})\subseteq \Ball{k-2}{r}$.   Now consider the vertices $$Z\coloneq V(R)\cup \{a,a_1,a_2,a'_1,a''_1\}\cup \{z,z'\}\cup \{b_1,b_2,b'_1,a'_2\}\cup V(Q)$$ It  can be verified that $\InducedSub{G}{Z}$ contains a theta, long prism, pyramid, or a broken wheel as an induced subgraph, which contradicts \Cref{prp:induced-truem}. 
			\end{subproof}
			
			Claims~\ref{clm:9.2}, \ref{clm:9.3}, and \ref{clm:9.4} imply that $\{a,a_2,b_2\}\subseteq S$. Since $\{aa_2,a_2b_2\}\in E(H)$, we have a contradiction due to \Cref{obs-trivial-1}.
		\end{proof}
		
		\vspace{-0.5cm}
		\subsection{Tree-width of $H$}\label{sec:tree-width}

		In this section, we shall show that  the output of \Cref{algo:embed} is a triangle-free universally signable graph and conclude using \Cref{prp:tree-width}.   We begin by observing that $H$ is theta-free.
		
		\begin{lemma}\label{lem:theta}
			The graph $H$ is theta-free.
		\end{lemma}
		\begin{proof}
			Assume $H$ contains a theta $T$ as an induced subgraph. Let $u,v\in V(T)$ be the vertices of degree three and let $P_1,P_2,P_3$ be the three $(u,v)$-paths of $T$. Without loss of generality, assume $\distance{G}{r}{u}\leq \distance{G}{r}{v}$ and $S$ be the cluster that contains $v$. 
			
			Suppose $u,v$ lies in different clusters. For each $i\in \{1,2,3\}$, there exist an edge $x_iy_i\in E(P_i)$ such that $x_i\in S$ and $y_i\in \ParentSet{r}{S}$. If $y_1,y_2,y_3$ are distinct vertices, we have a contradiction due to \Cref{obs-trivial-2}\ref{it:parent-deg}. Therefore, we can assume without loss of generality that $y_1=y_2=u$. This implies $x_1$ and $x_2$  are distinct vertices, and that $x_1\neq v, x_2\neq v$. However, now  $S$ contains three vertices from an induced cycle in $H$ (formed by $V(P_1)\cup V(P_2)$), which contradicts \Cref{lem:cycle}. 
			
			Hence $u,v$ lies in the same cluster and let $U$ be the cluster that contains $\ParentSet{r}{S}$. 
			Suppose there exists  two vertices $\{v_1,v_2\}\subseteq U$ such that $\{vv_1, vv_2\}\subseteq E(T)$. Without loss of generality, assume that $v_1\in V(P_1)$ and $v_2\in V(P_2)$. Let $C$ be the induced cycle in $H$ induced by $V(P_1)\cup V(P_2)$. Then, for each $i\in \{1,2\}$ there exist edges $x_iy_i\in E(P_i)$ such that $\{x_1,x_2\}\subseteq U$ and $\{y_1,y_2\}\subseteq S$. Applying \Cref{lem:cycle} on $C$, $S$, and $U$, we have that $\{x_1,x_2\}=\{v_1,v_2\}$ and $\{y_1,y_2\}=\{u,v\}$.  However, this contradicts \Cref{obs-trivial-3}. 
			Using similar arguments, we can show that there does not exist a cluster $S'$ in $G$ such that $\ParentSet{r}{S'} \subseteq S$ and $S'$ contains two vertices that are adjacent to $v$ in $T$. Hence, it must be the case that there exists a vertex $v_1\in S\setminus \{u,v\}$ such that $v,v_1\in V(T)$. Hence, $S$ contains three vertices $u,v,v_1$ from an induced cycle in $H$, which contradicts \Cref{lem:cycle}. 
		\end{proof}
		\begin{lemma}\label{lem:wheel}
			The graph $H$ is wheel-free.
		\end{lemma}
		\begin{proof}
			Assume that $H$ contains a wheel whose rim is $C$ and the center is $u$. Let $S$ be the cluster that contains $u$ and let $u_1,u_2,u_3$ be three neighbors of $u$ in $C$.  
			
			\revise{First, suppose $\{u1, u2, u3\} \subseteq S \cup P(r, S)$. Due to \Cref{obs-trivial-1}, at most one vertex from $\{u_1,u_2,u_3\}$ lies in $S$ and the remaining vertices lie in $P(r,S)$. But then the edge in $H[S]$ and the two edges in $H$ from $u$ to the vertices in  $P(r, S)$ contradict \Cref{obs-trivial-1}.}
			
			Hence, there exists a vertex $z\in \{u_1,u_2,u_3\}$ such that $\distance{G}{r}{z}=\distance{G}{r}{u}+1$ and $T$ be the cluster that contains $z$. Note that $\ParentSet{r}{T}\subseteq S$. Moreover, there must be a $z'\in \{u_1,u_2,u_3\}\setminus \{z\}$ such that $z'\notin T$. (Otherwise, $T$ will contain three vertices of $C$, which contradicts \Cref{lem:cycle}.) 
			\revise{Since $z'$ is adjacent to $u$, $z'$ belongs to $S \cup P(r, S)$ (in particular $d_G(r, z') \leq d_G(r, u)$).}
			 Then, there exists two distinct $(z,z')$-paths $Q_1\subseteq C,Q_2\subseteq C$, such that $V(Q_1)\cap V(Q_2)=\{z,z'\}$. Observe that for each $i\in \{1,2\}$, there must be an edge $x_iy_i\in E(Q_i)$ such that $x_i\in S$ and $y_i\in T$. Now, let $F\coloneq \{w\in \ParentSet{r}{T} \colon \exists v \in T, wv\in E(H)\}$. Observe that $\{x_1,x_2,u\}\subseteq F$, and hence $|F|\geq 3$, which contradicts \Cref{obs-trivial-2}\ref{it:parent-deg}. This implies $\{u_1,u_2,u_3\}\setminus \{z\} \subseteq S$, but this contradicts \Cref{obs-trivial-1}\ref{it:horizontal-a}. 
		\end{proof}
		
		\Cref{prp:tree-width,lem:triangle,lem:theta,lem:wheel} implies the following.
		
		\begin{lemma}\label{lem:tree-width}
			The tree-width of $H$ is at most two.
		\end{lemma}
		
		\vspace{-0.5cm}
		
		\subsection{Completion of the proof}\label{sec:complete}
		Recall that $G$ is a $K_{2,3}$-induced minor free graph, $r$ is an arbitrary vertex of $G$ (fixed in the beginning of \ref{sec:proof}), $H$ is the graph returned by \Cref{algo:embed}. Let $c \coloneq 2470$ and $u,v$ be any two vertices of $G$. Due to \Cref{prp:sipco}, the strong isometric path complexity of $G$ is at most $71$. Due to \Cref{lem:sipco-distort,lem:root-path-distort,lem:adj-distort},  $\distance{H}{u}{v} - \distance{G}{u}{v}\leq c$.  
		Due to \Cref{lem:triangle,lem:theta,lem:wheel} we know that $H$ is a universally signable graph, and therefore $K_{2,3}$-induced minor-free. Hence, the strong isometric path complexity of $H$ is at most $71$. Due to \Cref{lem:sipco-distort,clm:lem-G-1,clm:lem-G-2}, we have that $\distance{H}{u}{v} - \distance{G}{u}{v} \leq c$. Hence, $\dmod{\distance{H}{u}{v} - \distance{G}{u}{v}}\leq c$. Finally, \Cref{lem:tree-width} asserts that $H$ has tree-width at most two. 
		
		\revise{We argue that the running time of the \Cref{algo:embed} is $O(nm)$, where $n$ and $m$ are the number of vertices and edges of the input graph. Constructing a layering partition $T$ takes $O(n+m)$ time~\cite{chepoi2000note}. Then, \Cref{algo:embed} processes the clusters bottom-up. For each cluster $S$, it is possible to determine $\ParentSet{r}{S}$ by traversing the edges incident to the vertices of $S$. Hence, the total time to determine $\ParentSet{r}{S}$ for all $S$ is $O(m)$ time. Then for each cluster $S$, it is possible to determine whether $\InducedSub{G}{\ParentSet{r}{S}}$ is connected in $O(m)$ time. This is the most expensive step of our algorithm. Hence, the total running time of $\Cref{algo:embed}$ is $O(nm)$. This completes the proof.}
		
		\revise{We note that it is not clear if the total running time to determine whether $\InducedSub{G}{\ParentSet{r}{S}}$ is connected for all cluster $S$ can be further reduced. }

		\subsection{Remark on \Cref{cor:univ}}
		
		Any minimal cutset of universally signable graphs is either a clique or a stable set of size two (\Cref{prp:decom}).  Therefore,  when inputs of \Cref{algo:embed} are restricted to universally signable graphs the \emph{if} condition of \Cref{if-main} can be modified as follows:
		
		\begin{center}
			\emph{If $\InducedSub{G}{\ParentSet{r}{S}}$ contains at least three vertices or is an edge}
		\end{center} 
		
		This check can be done in $O(1)$-time, and the running time of \Cref{algo:embed} will be $O(n+m)$ where $n$ and $m$ are the number of vertices and edges of the input universally signable graph. Since the diameter of graphs of tree-width two can be found optimally in truly sub-quadratic time~\cite{cabello2018subquadratic}, we have the proof of \Cref{cor:univ}.
		
		\vspace{-0.25cm}
		\section{Conclusion}\label{sec:conclude}
		
		In this paper, combining the layering partition technique with the notion of strong isometric path complexity, we proved that $K_{2,3}$-induced minor-free graphs admit quasi-isometry with constant additive distortion to graphs with tree-width at most two. Minimizing the upper bound on the additive distortion is interesting. By tailoring the proof technique of Chakraborty et al. (Theorem 3, \cite{c2023isometric}) slightly and using \Cref{prp:ind2}, the bound on the strong isometric path complexity of $K_{2,3}$-induced minor free can be improved. This would lead to an improvement on the bound on the additive distortion (to below $1000$). However, we believe the actual value of the optimum additive distortion consists of single digits. Providing a linear time algorithm to map $K_{2,3}$-induced minor-free graphs with additive distortion to graphs with tree-width at most two would be interesting.

		Since $K_{2,r}$-induced minor-free graphs and more generally \emph{$K_{2,r}$-asymptotic minor-free} graphs~\cite{georgakopoulos2023graph} have bounded strong isometric path complexity, it would be interesting to study if these graph classes admit quasi-isometry with additive distortion to graphs with tree-width at most $r$. It would also be interesting to study if hereditary graph classes like  \emph{even-hole free} graphs~\cite{vuvskovic2010even}, \emph{pyramid-only} graphs, etc. admit quasi-isometry with additive distortion to graphs with tree-width at most two. A result of Albrechtsen et al. \cite{albrechtsen2024characterisation} implies that even-hole free and pyramid-only graphs admit quasi-isometry to graphs with tree-width at most two. However, both even-hole free graphs and pyramid-only graphs have unbounded strong isometric path complexity. Therefore, our technique cannot be used directly.
		
		\vspace{-0.25cm}
		
		\bibliographystyle{plain}
		\bibliography{references}
	\end{document}